\documentclass[a4paper, 11pt]{article}
\usepackage{amsmath,amssymb,esint,amscd,xspace,bm,fancyhdr,color,authblk,srcltx,fontenc,bbm}
\usepackage{hyperref}
\usepackage{cite}
\usepackage{graphicx,float,subfigure}

\newtheorem{theorem}{Theorem}[section]

\newtheorem{definition}[theorem]{Definition}

\newtheorem{lemma}[theorem]{Lemma}

\newtheorem{remark}[theorem]{Remark}
\newtheorem{example}[theorem]{Example}
\newenvironment{proof}[1][Proof]{\textbf{#1.} }{\hfill\rule{0.5em}{0.5em}}
{\catcode`\@=11\global\let\AddToReset=\@addtoreset
\AddToReset{equation}{section}

\AddToReset{theorem}{section}

\title{A new approach to convergence analysis of iterative models with optimal error bounds}

\author{Minh-Phuong Tran\footnote{Corresponding author.} \thanks{Applied Analysis Research Group, Faculty of Mathematics and Statistics, Ton Duc Thang University, Ho Chi Minh City, Vietnam; \texttt{tranminhphuong@tdtu.edu.vn}}, Thanh-Nhan Nguyen\thanks{Group of Analysis and Applied Mathematics,  Department of Mathematics, Ho Chi Minh City University of Education, Ho Chi Minh city, Vietnam; \texttt{nhannt@hcmue.edu.vn}}, Thai-Hung Nguyen, Tan-Phuc Nguyen, \\ Tien-Khai Nguyen, Cong-Duy-Nguyen Nguyen, Trung-Hieu Huynh\thanks{Department of Mathematics, Ho Chi Minh City University of Education, Ho Chi Minh city, Vietnam}}

\date{\today} 

\begin{document}
\maketitle

\begin{abstract}

In this paper, we study a new approach related to the convergence analysis of Ishikawa-type iterative models to a common fixed point of two non-expansive mappings in Banach spaces. The main novelty of our contribution lies in the so-called \emph{optimal error bounds}, which established some necessary and sufficient conditions for convergence and derived both the error estimates and bounds on the convergence rates for iterative schemes. Although a special interest here is devoted to the Ishikawa and modified Ishikawa iterative sequences, the theory of \emph{optimal error bounds} proposed in this paper can also be favorably applied to various types of iterative models to approximate common fixed points of non-expansive mappings.

\medskip

\medskip

\noindent {\emph{Keywords:} Non-expansive mappings;  Iterative processes; Optimal error bounds; Fixed point theory; Banach space; Convergence rate.} 

\medskip

\noindent{\emph{Mathematics Subject Classification.} 47H10; 47H09; 54H25; 65J05; 41A25.} 
\end{abstract} 

\tableofcontents

\section{Introduction and motivation}

A number of nonlinear equations of the form $F(x)=0$ are naturally formulated as fixed point problems. The fixed-point formulation of a nonlinear equation can be considered as
\begin{align}
\label{eq:fixed}
x=T(x),
\end{align}
where $T$ is the self-map of a set $X \subset \mathbb{R}$ that is endowed with a certain structure. These two equations are equivalent via $F(x)=x-T(x)$. In many cases, $T$ is assumed to be contractive or non-expansive mapping to ensure the existence of fixed points (see Browder~\cite{Browder} and Kirk~\cite{Kirk} for further reading). To the best of our knowledge, fixed point problems have a wide range of applications in different fields inside and beyond mathematics as approximation theory, optimization theory, control theory, finance, economics, medical sciences, engineering, machine learning, etc. For years, several iterative methods have been generated to approximate the fixed points of contractive and/or non-expansive mappings and this line of research has been developed by the works of many authors.

Let us briefly revisit some known iterative methods to approximate a fixed point of an operator. The simplest method to solve \eqref{eq:fixed} is the well-known Picard's iteration, which is commonly used, to find the sequence of successive approximations defined by an initial value $x_0 \in X$ and
\begin{align}\label{Picard}
x_{n+1} = T(x_n), \quad n \in \mathbb{N},
\end{align}
and in order to ensure the convergence, some strong contractive properties of the mapping $T$ must be included. It can be seen that Picard's iteration method is only convergent for an appropriate class of mappings and not convergent for other classes of mappings. In such cases, a classical iteration method was introduced by Mann in~\cite{Mann} to approximate the fixed points. It is a one-step iterative process defined as follows
\begin{align*}
x_{n+1} = (1-a_n)x_n + a_nT(x_n), \quad n \in \mathbb{N},
\end{align*}
where the initial guess $x_0 \in X$, the sequence $(a_n)_{n \in \mathbb{N}} \subset [0,1]$ such that $\displaystyle{\lim_{n \to \infty}{a_n}}=0$ and $\sum{a_n}=\infty$. It is worth mentioning that when $(a_n)$ is a constant sequence, the Mann iteration scheme reduces to the so-called Krasnoselskij iteration scheme, while $a_n \equiv 1$ for all $n \in \mathbb{N}$, we are back to Picard procedure. There is a large body of literature concerning the convergence of Mann iteration for different classes of nonlinear mappings and in various function spaces. However, it is known that the Mann iteration converges weakly to the fixed point of non-expansive mappings and consequently, several iterative schemes have been proposed to improve and obtain stronger convergence results (cf. \cite{D70, Halpern, Ishikawa, KX2005}). For instance, Ishikawa in \cite{Ishikawa} devised another iteration process 
\begin{align*}
x_{n+1} = (1-b_n)x_n + b_nT\big((1-a_n)x_n + a_nT(x_n)\big), \quad n \in \mathbb{N},
\end{align*}
or equivalently written via a two-step scheme
\begin{align}\label{Ishikawa-2}
\begin{cases} y_n &= (1-a_n)x_n + a_nT(x_n), \\ x_{n+1} &= (1-b_n)x_n + b_nT(y_n), \quad n \in \mathbb{N},
\end{cases}
\end{align}
with an initial guess $x_0 \in X$, $(a_n)_{n \in \mathbb{N}}$ and $(b_n)_{n \in \mathbb{N}}$ are two sequences of real numbers in $[0,1]$. It is clear from~\eqref{Ishikawa-2} that Mann iteration is a specific case when $a_n\equiv 0$ for all $n \in \mathbb{N}$. So far, it has been observed that Picard, Mann, and Ishikawa are three of the various authors whose contributions are of great value in the study of fixed-point iterative procedures. These methods as well as their modified iterative schemes are standard, and they have been studied and improved through the years. Besides, there has been considerable interest in iterative schemes of their types for finding  common fixed points for a pair of non-expansive mappings $(T_1,T_2)$. It was for example shown in~\cite{TT1998, YC2007} an iteration process defined as
\begin{align}\label{Ishi-T12}
\begin{cases}
x_0 \in X,\\
x_{n+1} = (1-b_n)x_n + b_nT_2\big((1-a_n)x_n + a_nT_1(x_n)\big), \quad n \in \mathbb{N},
\end{cases}
\end{align}
and we refer the reader to the references given there for interesting recent developments on iterative methods and convergence results to the common fixed-point. 

From the experimental point of view, it is known that the \emph{error} and \emph{rate of convergence} play an important role in checking the efficiency of iterative schemes. In the past years, the convergence rate, stability, and efficiency of different iterative methods have been discussed and improved by many authors (see for instance~\cite{Noor, KX2005, Berinde2008, CCK2003, SP2011, Popescu2007, Xue}). Since then, the construction of iterative methods for the fixed points of contractive/non-expansive and more general classes of mappings has been extensively investigated to obtain higher-order convergence rates, small errors, and higher efficiency coefficients. 

The main intent of this paper is to introduce a new standpoint to conclude the convergence of a multi-step fixed-point iteration method to a common fixed point of two given contractive/non-expansive mappings $T_1$ and $T_2$. In particular, we here define the so-called \emph{optimal error bounds} (OEBs) (see Definition~\ref{def:upper} and~\ref{def:lower}), and based on the analysis of optimal error bounds, it allows us to develop necessary and sufficient conditions for establishing the convergence of an iteration process. Let us be a bit more precise and explain how this concept comes into the convergence analysis. Let us consider a sequence $\{x_n\}_{n \in \mathbb{N}}$ generated by an iterative procedure involving two certain mappings $T_1$ and $T_2$ that converges to a common fixed point $x^*$ of $T_1, T_2$. As usual, estimating errors and convergence rate of an iterative sequence can be characterized from the following bound
\begin{align*}
\frac{\|x_n-x^*\|}{\|x_0-x^*\|} \le \texttt{Error bound}, \quad n \ge 1.
\end{align*}
However, the case of the \emph{optimal} error bound is more delicate and has been less investigated. To develop the idea with error bounds, we establish the theory of so-called \emph{optimal error bounds}, where the error estimate and the convergence rate will be derived of the following type:
\begin{align}
\label{eq:oeb}
\texttt{OLEB} \le  \frac{\|x_n-x^*\|}{\|x_0-x^*\|} \le \texttt{OUEB}, \quad n \ge 1,
\end{align}
where $x_0$ is an initial approximation. Here, the \emph{optimal lower error bounds} (OLEBs) and \emph{optimal upper error bounds} (OUEBs) are understood to be optimal in the sense that one cannot expect better error bounds without additional assumptions on two mappings $T_1$ and $T_2$ even in the general case. It should be said that from~\eqref{eq:oeb}, we devote special attention to the question of clarifying the sharp OEBs not only to obtain the convergence rate of an appropriate iterative scheme but also the comparison among various iterations. The estimation of the optimal error bounds is very important to obtain two issues: convergence and rate of convergence of the fixed-point iteration models. As such, the main novelty in this paper lies in the method of optimal error estimation and our contribution to the study of OEBs will improve and generalize many of the results in the literature. Moreover, it paves the way for the convergence analysis of an iterative algorithm for finding a common fixed point of two mappings. 

Building upon the approach from OEBs, we derive the necessary and sufficient conditions for the iterative schemes to converge to the common fixed point. In this study, the optimality is illustrated by the model example of Ishikawa-type iterative sequences and we analyze the convergence to a common fixed point of a pair of given mappings. We particularly formulate the convergence analysis of both the Ishikawa-type iteration and the its modified scheme (we refer to Section~\ref{sec:OEB_Ishikawa}). According to the new idea of OEBs, we are able to analyze the convergence of a fixed point iteration for contractive/non-expansive mappings in the framework of Banach spaces. 

The remainder of this paper is organized as follows. In Section~\ref{sec:pre}, we introduce some important definitions, some remarks, and preparatory lemmas which are required in the sequel. Section~\ref{sec:OEB_Ishikawa} motivates by treating two parts with the analysis of OEBs. The first one concerns the Ishikawa-type process and the second sets forth the modified Ishikawa-type process. In Section~\ref{sec:main}, we conduct a further analysis of the convergence of Ishikawa iterative methods with the idea of OEBs. More precisely, we focus our attention on convergence rate estimates and comparison results between the rate of convergences between models. Section~\ref{sec:tests} is devoted to some numerical tests, to show the efficiency of the proposed optimal error bounds. We therein illustrate the convergence rate and the comparative study on the convergence rate via several examples. In the last section, we draw some conclusions, discuss and point out possible directions for future research.

\section{Preliminaries and Definition of OEBs}
\label{sec:pre}

In this preparatory section, we provide some preliminaries,  discuss the idea of so-called \emph{optimal error bounds}, and some technical lemmas that will be used in the sequel. Throughout the paper, we always assume $\mathbb{X}$ is a Banach space in $\mathbb{R}$.

\begin{definition}[Non-expansive mapping]
\label{def-Ne}
For each non-empty subset $\Omega$ of $\mathbb{X}$, a mapping $T: \Omega \to \Omega$ is said that to be non-expansive if there exists $\alpha \in [0,1]$ such that
\begin{align}\label{ineq-def:Ne}
\|T(x) - T(y)\| \le \alpha \|x-y\|, \quad \text{for all} \ x,y \in \Omega.
\end{align}
In this case, we denote $T \in \mathbf{Ne}_{\alpha}(\Omega)$. And when $\alpha<1$, we say that $T$ is a $\alpha$-contraction and write $T \in \mathbf{Co}_{\alpha}(\Omega)$.
\end{definition}

\begin{definition}[Common fixed point of two mappings]
Let $T_1 \in \mathbf{Ne}_{\alpha_1}(\Omega)$, $T_2 \in \mathbf{Ne}_{\alpha_2}(\Omega)$ for some $\alpha_1, \alpha_2 \in [0,1]$. Then, $x^* \in \Omega$ is called a common fixed point of two mappings $T_1, T_2$ if 
$$T_1(x^*) = T_2(x^*) = x^*.$$ 
In this case, we write $x^* \in \mathcal{F}(T_1, T_2)$, the set of all common fixed points of $T_1$ and $T_2$.
\end{definition}

\begin{definition}[Comparison of convergence rates]
\label{def-conv-rate}
Let $(u_n)_{n \in \mathbb{N}^*}$ and $(x_n)_{n \in \mathbb{N}^*}$ be two sequences in the Banach space~$\mathbb{X}$. Assume that both sequences $(u_n)$ and $(x_n)$ converge to $p$ in $\mathbb{X}$ as $n$ tends to infinity. We say that $(u_n)$ converges to $p$ faster than $(x_n)$ if 
\begin{align}\label{faster}
\lim\limits_{n \to \infty} \mathcal{R}(u_n,x_n,p) = 0,
\end{align}
where $\mathcal{R}(u_n,x_n,p)$ denotes the ratio between $\|u_{n} - p\|$ and $\|x_{n} - p\|$, i.e.
\begin{align}\label{Ra}
\mathcal{R}(u,x,p) := \begin{cases} \displaystyle\frac{\|u - p\|}{\|x - p\|}, &\ \mbox{ if } x \neq p, \\ 1, &\ \mbox{ if } u \neq x = p, \\ 0, &\ \mbox{ if } u = x = p. \end{cases}
\end{align}
\end{definition}

At this stage, we shall give the main notion for our work: the optimal error bounds of an iterative process. First, let $\Omega$ be a closed convex subset in Banach space $\mathbb{X}$. We consider an iterative process $(\mathcal{P})$ for approximating the common fixed point $x^* \in \mathcal{F}(T_1, T_2)$ of two non-expansive mappings $T_1 \in \mathbf{Ne}_{\alpha_1}(\Omega)$ and $T_2 \in \mathbf{Ne}_{\alpha_2}(\Omega)$, for $\alpha_1, \alpha_2 \in [0,1]$ as follows:
\begin{align}\tag{$\mathcal{P}$}
\label{P}
\begin{cases} x_0 \in \Omega, \\ x_{n+1} = P_{T_1,T_2}(x_n), \quad n \in \mathbb{N},\end{cases}
\end{align}
where $P_{T_1,T_2}: \Omega \to \Omega$ and $P_{T_1,T_2}(x^*)=x^*$. For example, the iterative process~\eqref{Ishi-T12} reads the function $P_{T_1,T_2}$ of problem~\eqref{P} as follows
\begin{align*}
P_{T_1,T_2}(x) = (1-b_n)x + b_nT_2\big((1-a_n)x + a_nT_1(x)\big), \quad x \in \Omega,
\end{align*}

\begin{definition}[Optimal upper error bound (OUEB)] 
\label{def:upper}
A non-negative sequence $(\mathcal{U}_n)_{n\in \mathbb{N}}$ is said to be the upper error bound sequence of an iterative process $(\mathcal{P})$ if the following inequality
\begin{align}\label{upper-bound}
\|x_{n+1} - x^*\| \le \mathcal{U}_n \|x_0 - x^*\|, \quad \forall n \in \mathbb{N},
\end{align}
holds for any pair of non-expansive mappings $T_1 \in \mathbf{Ne}_{\alpha_1}(\Omega)$ and $T_2 \in \mathbf{Ne}_{\alpha_2}(\Omega)$ satisfying $x^* \in \mathcal{F}(T_1,T_2)$. Moreover, we say that the sequence $(\mathcal{U}_n)_{n\in \mathbb{N}}$ is optimal if there exist $T_1 \in \mathbf{Ne}_{\alpha_1}(\Omega)$ and $T_2 \in \mathbf{Ne}_{\alpha_2}(\Omega)$ such that the equality in~\eqref{upper-bound} holds, i.e.
\begin{align*}
\|x_{n+1} - x^*\|  = \mathcal{U}_n \|x_0 - x^*\|, \quad \text{ for every } n \in \mathbb{N}.
\end{align*} 
\end{definition}

We shall give an example to the reader how to determine an optimal upper error bound for the Picard iteration.
\begin{example}
Assume that $\alpha \in (0,1)$ and $(x_n)$ is the Picard iteration sequence defined as in~\eqref{Picard} with initial value $x_0 \in \Omega$. We consider $T \in  \mathbf{Co}_{\alpha}(\Omega)$ admits a fixed point $x^*$. Then, it is easy to verify that 
\begin{align*}
\|x_{n+1} - x^*\| \le \alpha^n \|x_0 - x^*\|, \quad \mbox{for every } n \in \mathbb{N}.
\end{align*}
Here, for each $x^* \in \Omega$, we further consider a mapping $T: \Omega \to \Omega$ as
\begin{align*}
T(x) = \alpha x + (1-\alpha)x^*, \quad x \in \Omega.
\end{align*}
Clearly, $T \in  \mathbf{Co}_{\alpha}(\Omega)$ and $T(x^*) = x^*$, and therefore one has
\begin{align*}
\|x_{n+1} - x^*\| = \alpha^n \|x_0 - x^*\|, \quad \mbox{for every } n \in \mathbb{N}.
\end{align*}
From this fact, we may conclude that the sequence $(\mathcal{U}_n)_{n\in \mathbb{N}}$ defined by
\begin{align}\label{U-Picard}
\mathcal{U}_n = \alpha^n, \quad n \in \mathbb{N},
\end{align}
is an optimal upper error bound for the Picard iteration~\eqref{Picard}. One can see that upper error bound \eqref{U-Picard} can not be improved without additional conditions of contraction $T$.
\end{example}

It is remarkable from Definition~\ref{def:upper} that an iteration converges whenever the sequence of upper error bounds tends to zero. Furthermore, an interesting point is that the optimal upper error bound depends only on the contraction constant $\alpha$ and is independent of $T$ and $\Omega$.

Also, we follow the same strategy for the definition of optimal lower error bound.

\begin{definition}[Optimal lower error bound (OLEB)]
\label{def:lower}
A non-negative sequence $(\mathcal{L}_n)_{n\in \mathbb{N}}$ is said to be a lower error bound of an iterative process $(\mathcal{P})$ if 
\begin{align}
\label{lower-bound}
\|x_{n+1} - x^*\|  \ge \mathcal{L}_n \|x_0 - x^*\|, \quad \forall n \in \mathbb{N},
\end{align}
for any pair of non-expansive mappings $T_1 \in \mathbf{Ne}_{\alpha_1}(\Omega)$, $T_2 \in \mathbf{Ne}_{\alpha_2}(\Omega)$ satisfying $x^* \in \mathcal{F}(T_1,T_2)$. Moreover, we say that  the sequence $(\mathcal{L}_n)_{n\in \mathbb{N}}$ is optimal if there exist $T_1 \in \mathbf{Ne}_{\alpha_1}(\Omega)$ and $T_2 \in \mathbf{Ne}_{\alpha_2}(\Omega)$ such that the equality of ~\eqref{lower-bound} holds. 
\end{definition}

It is natural to expect that the idea of OEBs in Definition~\ref{def:upper} and~\ref{def:lower} allows us to conclude the convergence of an iterative sequence, and further compare the rate of convergence among various iterative schemes. Therefore, one would like to better understand the OEBs to derive necessary and sufficient conditions for the convergence of an iteration process to the common fixed point of non-expansive mappings. Along with the estimation of OEBs, the convergence analysis of numerical methods becomes more rigorous and efficient. We shall present the convergence of Ishikawa-type process with the study of OEBs in the next section, and perform some computational experiments in Section~\ref{sec:tests}.

In what follows, we stress that we shall use the $\approx$ notation mainly to highlight the equivalence of sequences. By $u_n \approx v_n$ we mean that there exist two constants $c_1, c_2>0$ not depending on $n$ such that 
\begin{align}\label{def-approx}
c_1 u_n \le v_n \le c_2 u_n, \quad \mbox{ for every } n \in \mathbb{N}.
\end{align}

Moreover, if either both series $\displaystyle{\sum_{k=0}^{\infty} a_k}$ and $\displaystyle{\sum_{k=0}^{\infty} b_k}$ converge or both diverge, we will often write 
\begin{align} \label{eq:equiv}
\sum_{k=0}^{\infty} a_k \sim \sum_{k=0}^{\infty} b_k.
\end{align}

Hereafter, we prove the following useful technical lemma that will be used in our later sections. 
\begin{lemma}\label{lem:seri}
Let $(a_k)_{k \in \mathbb{N}}$ be a sequence in $[0,1]$ and  $(u_k)_{k \in \mathbb{N}}$ be a bounded sequence such that 
$$1 - a_k u_k >0 \ \mbox{ for every } \ k \in \mathbb{N}.$$ 
Then, there holds
\begin{align}\label{seri-sim}
\sum^{\infty}_{k=0}\frac{a_k}{1-u_ka_k} \sim \sum^{\infty}_{k=0} a_k.
\end{align}
\end{lemma}
\begin{proof}
First, we obviously have
$$\frac{a_k}{1 - a_k u_k} \ge a_k, \quad \mbox{ for every } k \ge 0,$$
and therefore, if the series $\displaystyle \sum_{k=0}^{\infty} a_k$ diverges then $\displaystyle \sum_{k=0}^{\infty} \frac{a_k}{1 - a_ku_k}$ also diverges. On the other hand, when the series $\displaystyle \sum_{k=0}^{\infty} \frac{a_k}{1 - a_ku_k}$ diverges, we need to show that the series $\displaystyle \sum_{k=0}^{\infty} a_k$ also diverges. Indeed, if $\displaystyle \sum_{k=0}^{\infty} a_k$ converges, then one has $\lim\limits_{k\to \infty} a_k = 0$. Then, if it assumes that $\displaystyle \sum_{k=0}^{\infty} \frac{a_k}{1 - a_ku_k}$ diverges, there exist infinite numbers of $k \in \mathbb{N}$ such that 
$$\frac{a_k}{1 - a_ku_k} \neq 0$$ 
and it means that there exists a sub-sequence $(k_j)_{j\in \mathbb{N}} \subset \mathbb{N}$ such that 
$$\frac{a_{k_j}}{1 - a_{k_j}u_{k_j}} > 0, \quad \mbox{ for every } j \in \mathbb{N}.$$ 
On the other hand, it is easy to see that
$$\lim\limits_{j\to \infty} \frac{\frac{a_{k_j}}{1 - a_{k_j}u_{k_j}}}{a_{k_j}}  = \lim\limits_{j\to \infty} \frac{1}{1 - a_{k_j}u_{k_j}} = 1.$$
This implies that both series $\displaystyle \sum_{j=0}^{\infty} \frac{a_{k_j}}{1 - a_{k_j}u_{k_j}}$ and $\displaystyle \sum_{j=0}^{\infty} a_{k_j}$ converge, that is a contraction. Hence, two series in \eqref{seri-sim} are equivalent in the sense of~\eqref{eq:equiv}. 
\end{proof}

\section{Ishikawa-type iterative processes with OEBs}
\label{sec:OEB_Ishikawa}

In this section, we provide an illustration of convergence analysis for Ishikawa-type iteration process based on the idea of optimal upper and lower error bounds.

\subsection{The Ishikawa iteration process}

In what follows, let $\alpha_1, \alpha_2 \in [0,1]$ such that $0<\alpha_1+\alpha_2<2$ and $(a_n)_{n \in \mathbb{N}}$, $(b_n)_{n \in \mathbb{N}}$ be two sequences in $[0,1]$. The Ishikawa iterative process (I), associated with two non-expansive mappings $T_1 \in \mathbf{Ne}_{\alpha_1}(\Omega)$ and $T_2 \in \mathbf{Ne}_{\alpha_2}(\Omega)$, that admitted a common fixed point $x^* \in \mathcal{F}(T_1, T_2)$, is generated by an initial point $x_0 \in \Omega$ and
\begin{align}\tag{I}\label{xn-I}
x_{n+1} = (1-b_n)x_n + b_nT_2\big((1-a_n)x_n + a_nT_1(x_n)\big), \quad n \in \mathbb{N}.
\end{align}

In the next theorem, by determining the OUEB sequence for the Ishikawa process~\eqref{xn-I}, we establish some new necessary and sufficient conditions for the convergence of the OUEB sequence. As we shall see in the proof, the study of OEBs is important and useful to conclude the convergence of the iterative process~\eqref{xn-I} and prove a sharp convergence rate in general.

\begin{theorem}
\label{theo-IA}
Under the assumptions on $\alpha_1, \alpha_2, (a_n)$ and $(b_n)$ as above, assume that $\Omega$ is a convex subset in Banach space $\mathbb{X}$. Then, the OUEB sequence $(\mathcal{U}_n^{\mathrm{I}})_{n \in \mathbb{N}}$ of Ishikawa process~\eqref{xn-I} is defined by:
\begin{align}\label{U-IA}
\mathcal{U}_n^{\mathrm{I}} & = \displaystyle \prod_{k=0}^n \big(1- b_k + \alpha_2b_k(1 - a_k + \alpha_1a_k)\big), \quad n \in \mathbb{N}.
\end{align}
Moreover, the following statement is true 
\begin{align}\label{U-IA-0}
\lim_{n \to \infty} \mathcal{U}_n^{\mathrm{I}} = 0 \Longleftrightarrow \displaystyle{(1-\alpha_2)\sum_{k=0}^{\infty} b_k + \chi_{\{\alpha_2=1\}} (1-\alpha_1)\sum_{k=0}^{\infty} a_k b_k= +\infty},
\end{align}
where $\chi_{\{\alpha_2=1\}}$ is defined by
\begin{align*}
\chi_{\{\alpha_2=1\}}:= \begin{cases} 1, & \mbox{ if } \, \alpha_2 = 1, \\ 0, & \mbox{ if } \, \alpha_2 \neq 1. \end{cases}
\end{align*}
\end{theorem}
\begin{proof}
In the first step, for a pair of non-expansive mappings $T_1 \in \mathbf{Ne}_{\alpha_1}(\Omega)$ and $T_2 \in \mathbf{Ne}_{\alpha_2}(\Omega)$ that admitted a common fixed point $x^* \in \mathcal{F}(T_1, T_2)$, the Ishikawa iterative sequence~\eqref{xn-I} can be rewritten as
\begin{align}\label{xn-I-yn}
\begin{cases} x_0 \in \Omega, \\ y_{n} = (1-a_n)x_n + a_nT_1(x_n), \\ x_{n+1} = (1-b_n)x_n + b_nT_2(y_n), \quad n \in \mathbb{N}.\end{cases}
\end{align} 
As $T_1 \in \mathbf{Ne}_{\alpha_1}(\Omega)$ and $T_1(x^*) = x^*$, it derives that
\begin{align}\label{est-I1-21}
\|y_n - x^*\| & = \|(1-a_n)x_n + a_nT_1(x_n) - (1-a_n)x^* - a_nT_1(x^*)\| \notag\\
& \le (1-a_n)\|x_n - x^*\| + a_n\|T_1(x_n) - T_1(x^*)\| \notag\\
& \le (1-a_n+\alpha_1 a_n)\|x_n-x^*\|,
\end{align}
for $n \in \mathbb{N}$. At this stage, exploiting the fact that $T_2 \in \mathbf{Ne}_{\alpha_2}(\Omega)$, we readily obtain
\begin{align}\label{est-I1-22}
\|x_{n+1} - x^*\| & = \|(1-b_n)x_n + b_nT_2(y_n) - (1-b_n)x^* - b_nT_2(x^*)\| \notag \\
& \le (1-b_n)\|x_n-x^*\| + \alpha_2 b_n\|y_n-x^*\|.
\end{align}
As a consequence of~\eqref{est-I1-21} and~\eqref{est-I1-22}, we infer that
\begin{align*}
\|x_{n+1} - x^*\| & \le \big(1-b_n + \alpha_2 b_n (1-a_n+\alpha_1 a_n)\big)\|x_n-x^*\|, \quad \forall n \in \mathbb{N}.
\end{align*}
By induction, we can easily confirm that 
$$\|x_{n+1} - x^*\| \le \mathcal{U}_n^{\mathrm{I}}\|x_0-x^*\|, \quad \mbox{ for every } n \in \mathbb{N},$$ 
where the sequence $\mathcal{U}_n^{\mathrm{I}}$ is defined as in~\eqref{U-IA}. This ensures that the sequence $\mathcal{U}_n^{\mathrm{I}}$ is an upper error bound of iterative process~\eqref{xn-I}. \\

Next, to show that this sequence is optimal, we shall need to find two mappings $T_1$ and $T_2$ to make sure the equality signs. First, let us fix a point $x^* \in \Omega$ and consider two non-expansive mappings $T_1, T_2: \Omega \to \Omega$ as follows
\begin{align*}
T_1(x) = \alpha_1 x + (1-\alpha_1)x^*, \quad T_2(x) = \alpha_2 x + (1-\alpha_2)x^*, \quad x \in \Omega.
\end{align*}
It is easy to see that $T_1 \in \mathbf{Ne}_{\alpha_1}(\Omega)$ and $T_2 \in \mathbf{Ne}_{\alpha_2}(\Omega)$ having a unique common fixed point $x^* \in \mathcal{F}(T_1, T_2)$. Moreover,  one can check that $T_1(x) - x^* = \alpha_1 (x-x^*)$ and $T_2(x) - x^* = \alpha_2 (x-x^*)$, for all $x \in \Omega$. For this reason, we are able to estimate
\begin{align*}
\|x_{n+1}-x^*\| & = \left\|(1-b_n)x_n + b_nT_2\big((1-a_n)x_n + a_nT_1(x_n)\big)-x^*\right\| \\
& = \left\|(1-b_n)(x_n-x^*) + \alpha_2 b_n\big((1-a_n)(x_n-x^*) + \alpha_1 a_n (x_n-x^*)\big)\right\| \\
& = \big(1-b_n + \alpha_2 b_n (1-a_n+\alpha_1 a_n)\big)\|x_n-x^*\|,
\end{align*}
which leads to 
\begin{align*}
\|x_{n+1} - x^*\|  = \mathcal{U}_n^{\mathrm{I}}\|x_0-x^*\|, \quad \forall n \in \mathbb{N},
\end{align*}
and we therefore conclude that $\mathcal{U}_n^{\mathrm{I}}$ in~\eqref{U-IA} defines an OUEB of process~\eqref{xn-I}.\\ 

In the next step, for the proof of~\eqref{U-IA-0}, let us now distinguish two cases: 
\begin{enumerate}
\item [(i)] Case 1: if  $\alpha_2<1$, the preceding statement \eqref{U-IA-0} becomes
\begin{align}
\label{U-IA-1st}
\lim_{n \to \infty} \mathcal{U}_n^{\mathrm{I}} = 0 \Longleftrightarrow \displaystyle{\sum_{k=0}^{\infty} b_k = +\infty}.
\end{align}
\item [(ii)] Else, if $\alpha_2=1$, then \eqref{U-IA-0} can be rewritten as
\begin{align}
\label{U-IA-2sd}
\lim_{n \to \infty} \mathcal{U}_n^{\mathrm{I}}  = 0 \Longleftrightarrow \displaystyle{\sum_{k=0}^{\infty} a_kb_k = +\infty}.
\end{align}	 
\end{enumerate}

For the first case, without loss of generality, we may assume that $\mathcal{U}_n^{\mathrm{I}} >0$ for all $n \in \mathbb{N}$. Then, for every $n \in \mathbb{N}$, it allows us to write $\log \mathcal{U}_n^{\mathrm{I}}$ as follows
\begin{align}\notag
\log \mathcal{U}_n^{\mathrm{I}} = \sum_{k=0}^n \log \big(1- b_k + \alpha_2b_k(1 - a_k + \alpha_1a_k)\big).
\end{align}
Applying the fundamental inequality
\begin{align}\label{ln-ineq}
\frac{x-1}{x} \le \log x \le x-1, \quad \mbox{ for all } x>0,
\end{align}
yields that
\begin{align}\label{est-U-IA-0}
\sum_{k=0}^n \frac{- b_k + \alpha_2b_k(1 - a_k + \alpha_1a_k)}{1- b_k + \alpha_2b_k(1 - a_k + \alpha_1a_k)} \le \log \mathcal{U}_n^{\mathrm{I}} \le \sum_{k=0}^n \big(- b_k + \alpha_2b_k(1 - a_k + \alpha_1a_k)\big).
\end{align}
It should be noted that $\alpha_1 \le 1 - a_k + \alpha_1a_k \le 1$ and $1- b_k + \alpha_2b_k(1 - a_k + \alpha_1a_k) \ge \alpha_1\alpha_2$ for all $k \ge 1$. As such, from \eqref{est-U-IA-0} we obtain
\begin{align}\label{est-U-IA}
\dfrac{\alpha_1\alpha_2 - 1}{\alpha_1\alpha_2}\sum_{k=0}^n b_k \le \log \mathcal{U}_n^{\mathrm{I}} \le (\alpha_2 - 1)\sum_{k=0}^n b_k, \quad \mbox{ for all } n\ge 1.
\end{align}
Having arrived at this stage, since $\alpha_2-1<0$, the estimation in \eqref{est-U-IA} allows us to conclude \eqref{U-IA-1st}. \\

For the second case when $\alpha_2=1$, in a similar fashion as the previous case, we write 
\begin{align}\notag
\sum_{k=0}^n \frac{- b_k + b_k(1 - a_k + \alpha_1a_k)}{1- b_k + b_k(1 - a_k + \alpha_1a_k)} &\le \log \mathcal{U}_n^{\mathrm{I}}  \le \sum_{k=0}^n (- b_k + b_k(1 - a_k + \alpha_1a_k)),
\end{align}
that is equivalent to
\begin{align}\label{est-U-IA1b} 
(\alpha_1 - 1 )\sum_{k=0}^n \frac{a_kb_k}{1- (1-\alpha_1)a_kb_k} &\le \log \mathcal{U}_n^{\mathrm{I}} \le (\alpha_1 - 1)\sum_{k=0}^n a_kb_k.
\end{align}
Thanks to Lemma~\ref{lem:seri}, we readily obtain
\begin{align*}
\sum_{k=0}^{\infty} \frac{a_kb_k}{1- (1-\alpha_1)a_kb_k} \sim \sum_{k=0}^{\infty} a_kb_k.
\end{align*}
Combining this fact with \eqref{est-U-IA1b}, we conclude that the statement in \eqref{U-IA-2sd} holds. The proof is then complete.   
\end{proof}

\begin{remark}
We remark that the condition that $\alpha_1 + \alpha_2 \in (0,2)$ arises naturally. Indeed, if $\alpha_1 = \alpha_2 = 0$ or $\alpha_1 = \alpha_2 = 1$, then the convergence becomes trivial. 
\end{remark}

As for now, we can see that if the OUEB sequence converges to zero, then the Ishikawa iterative sequence~\eqref{xn-I} converges to a common fixed point of two non-expansive mappings. Nevertheless, to characterize the convergence rate and the order of convergence of an iterative scheme, the OUEB sequence is not sufficient.  Here, the study of the OLEB sequence will come into play to conclude the convergence of an iterative process analytically and empirically.  The next theorem will provide us with the OLEB sequence and its crucial role in the convergence of Ishikawa process~\ref{xn-I}. From a numerical analysis viewpoint, this theorem also yields the necessary and sufficient conditions on initial data for the sequence of OLEB to converge to zero and therefore it allows us to establish the sequential convergence and convergence rate of a sequence of iterates proposed by Ishikawa in~\eqref{xn-I}. 

\begin{theorem}\label{theo-IB}
Under the assumptions on $\alpha_1, \alpha_2, (a_n)$ and $(b_n)$ as above, assume that $\Omega$ is a convex subset in Banach space $\mathbb{X}$ and 
\begin{align}\label{A-theo-IB}
1-b_k - \alpha_2 b_k(1 - a_k + \alpha_1a_k) > 0, \quad \mbox{ for every } \ k \in \mathbb{N}.
\end{align} 
Then, the OLEB sequence $(\mathcal{L}_n^{\mathrm{I}})_{n \in \mathbb{N}}$ of Ishikawa process~\eqref{xn-I} is defined by
\begin{align}\label{L-IB}
\mathcal{L}_n^{\mathrm{I}} & = \prod_{k=0}^n \big(1- b_k - \alpha_2b_k(1 - a_k + \alpha_1a_k)\big), \quad n \in \mathbb{N}.
\end{align}
Moreover, the following statement holds true
\begin{align}\label{L-IB-1}
\lim_{n \to \infty} \mathcal{L}_n^{\mathrm{I}} = 0 \Longleftrightarrow \sum^{\infty}_{k=0}b_k=+\infty.
\end{align}
\end{theorem}
\begin{proof}
Let $(\mathcal{L}_n^{\mathrm{I}})$ be a sequence of positive numbers defined by~\eqref{L-IB}. We firstly show that $\mathcal{L}_n^{\mathrm{I}}$ is a lower error bound of iterative process~\eqref{xn-I}. More precisely, we need to prove that the following inequality
\begin{align}\label{est-L-IB-10}
\|x_{n+1}-x^*\| \ge \mathcal{L}_n^{\mathrm{I}} \|x_0 - x^*\|, \quad \mbox{ for all } n \in\mathbb{N},
\end{align}
holds for every non-expansive mappings $T_1 \in \mathbf{Ne}_{\alpha_1}(\Omega)$ and $T_2 \in \mathbf{Ne}_{\alpha_2}(\Omega)$ admitted a common fixed point $x^* \in \mathcal{F}(T_1, T_2)$. Indeed, if $(x_n)_{n\in\mathbb{N}}$ is defined as in \eqref{xn-I-yn}, it enables us to modify the estimate~\eqref{est-I1-22} as follows
\begin{align}
\|x_{n+1} - x^*\| & = \|(1-b_n)x_n + b_nT_2(y_n) - (1-b_n)x^* - b_nT_2(x^*)\| \notag \\
& \ge (1-b_n)\|x_n-x^*\| - b_n\|T_2(y_n)-T_2(x^*)\|, \notag 
\end{align}
for every $n \in \mathbb{N}$. Combining with assumption $T_2 \in \mathbf{Ne}_{\alpha_2}(\Omega)$, it deduces that
\begin{align}
\|x_{n+1} - x^*\| & \ge \big(1-b_n - \alpha_2 b_n(1-a_n+\alpha_1 a_n)\big)\|x_n-x^*\|, \notag
\end{align}
and it leads to \eqref{est-L-IB-10} by induction on $n$. We next show that the sequence of lower error bounds $(\mathcal{L}_n^{\mathrm{I}})_{n \in \mathbb{N}}$ is optimal by choosing $x^* \in \Omega$ and two non-expansive mappings $T_1 \in \mathbf{Ne}_{\alpha_1}(\Omega)$, $T_2 \in \mathbf{Ne}_{\alpha_2}(\Omega)$ defined respectively by
\begin{align*}
T_1(x) = \alpha_1 x + (1-\alpha_1)x^*, \quad T_2(x) = -\alpha_2 x+(1+\alpha_2)x^*, \quad x \in \Omega.
\end{align*}
Analogously to the proof of Theorem~\ref{theo-IA}, it is easy to check that \eqref{est-L-IB-10} is valid. Therefore, the sequence $(\mathcal{L}_n^{\mathrm{I}})_{n \in\mathbb{N}}$ is optimal in process~\eqref{xn-I}.

Having arrived at this stage, we prove the equivalent statement in \eqref{L-IB-1}. For simplicity of notation, let us define 
$$A_k = 1+\alpha_2(1-a_k+\alpha_1 a_k), \quad k \in \mathbb{N}.$$ 
Thanks to \eqref{ln-ineq} again, one obtains that
\begin{align}\label{est-L-IB-0}
-\sum^{n}_{k=0}\frac{b_k A_k}{1-b_k A_k}\le  \log \mathcal{L}_n^{\mathrm{I}} \le -\sum^{n}_{k=0} b_k A_k.
\end{align} 
Combining between \eqref{est-L-IB-0} and the fact $1+\alpha_1 \alpha_2 \le A_k \le 1 + \alpha_2$ for every $k\ge 0$, it follows that
\begin{align}\label{est-L-IB-01}
-(1 + \alpha_2)\sum^{n}_{k=0}\frac{b_k}{1-b_k A_k}\le \log \mathcal{L}_n^{\mathrm{I}} \le -(1+\alpha_1\alpha_2)\sum^{n}_{k=0} b_k.
\end{align}
On the other hand, the assertion of Lemma~\ref{lem:seri} gives us
 $$\sum^{\infty}_{k=0}\frac{b_k}{1-b_k A_k} \sim \sum^{\infty}_{k=0} b_k,$$ 
which allows us to conclude \eqref{L-IB-1} from \eqref{est-L-IB-01}. The proof is therefore complete.
\end{proof}

\subsection{The modified Ishikawa iteration process}

$\hspace{.4cm}$ We now consider a modified Ishikawa iterative process under the same hypotheses as the preceding scheme. Given two constants $\alpha_1, \alpha_2 \in [0,1]$ such that $0<\alpha_1+\alpha_2<2$ and two sequences $(a_n)_{n \in \mathbb{N}}$, $(b_n)_{n \in \mathbb{N}}$ in the interval $[0,1]$. Let us consider the Ishikawa iterative process (IM) defined by: for every non-expansive mappings $T_1 \in \mathbf{Ne}_{\alpha_1}(\Omega)$ and $T_2 \in \mathbf{Ne}_{\alpha_2}(\Omega)$ admitted a common fixed point $x^* \in \mathcal{F}(T_1, T_2)$, it determines a sequence $(x_n)$ as follows
\begin{align}\label{xn-IM}
x_{n+1} = (1-b_n)\big((1-a_n)x_n + a_nT_1(x_n)\big) + b_nT_2\big((1-a_n)x_n + a_nT_1(x_n)\big), \quad n \in \mathbb{N},
\end{align}
with the initial point $x_0 \in \Omega$. Similar to above argument of Ishikawa process~\eqref{xn-I}, we are allowed to rewrite \eqref{xn-IM} as a two-iteration process
\begin{align}\tag{IM}\label{xn-IM-yn}
\begin{cases} x_0 \in \Omega,\\ y_{n} = (1-a_n)x_n + a_nT_1(x_n), \\ x_{n+1} = (1-b_n)y_n + b_nT_2(y_n), \quad n \in \mathbb{N}.\end{cases}
\end{align}

\begin{theorem}\label{theo-IM-A}
Let $\alpha_1, \alpha_2 \in [0,1]$ such that $0<\alpha_1+\alpha_2<2$ and $(a_n)_{n \in \mathbb{N}}$, $(b_n)_{n \in \mathbb{N}}$ be two sequences in the interval $[0,1]$. Assume that $\Omega$ is a convex subset in Banach space $\mathbb{X}$. Then, the OUEB sequence $(\mathcal{U}_n^{\mathrm{IM}})_{n \in \mathbb{N}}$ of the modified Ishikawa iterative process~\eqref{xn-IM} is defined by
\begin{align}\label{U-IM-A}
& \mathcal{U}_n^{\mathrm{IM}} = \prod\limits_{k=0}^n (1-a_k + \alpha_1 a_k)(1-b_k + \alpha_2 b_k), \quad  n \in \mathbb{N}.
\end{align}
Moreover, the following statement is true 
\begin{align}\label{U-IM-A-1st}
\lim_{n \to \infty} \mathcal{U}_n^{\mathrm{IM}} = 0 \Longleftrightarrow \displaystyle{(1-\alpha_1)\sum_{k=0}^{\infty} a_k + (1-\alpha_2)\sum_{k=0}^{\infty} b_k  = +\infty}.
\end{align}
\end{theorem}
\begin{proof} 
To show that $(\mathcal{U}_n^{\mathrm{IM}})_{n \in \mathbb{N}}$ defined in \eqref{U-IM-A} is the OUEB sequence of the modified Ishikawa iterative process~\eqref{xn-IM}, we use an approach similar to the proof in Theorem~\ref{theo-IA}. It only varies in the following estimate 
\begin{align}
\|x_{n+1} - x^*\| & = \|(1-b_n)y_n + b_nT_2(y_n) - (1-b_n)x^* - b_nT_2(x^*)\| \notag \\
& \le (1-b_n+\alpha_2 b_n)\|y_n-x^*\| \notag\\
& \le (1-a_n+\alpha_1 a_n)(1-b_n+\alpha_2 b_n)\|x_n-x^*\|.\notag
\end{align}
To finish, let us now prove~\eqref{U-IM-A-1st}. Without loss of generality, we may assume that $\mathcal{U}_n^{\mathrm{IM}}>0$, for every $n \in \mathbb{N}$, and hence, it allows us to present $\log\mathcal{U}_n^{\mathrm{IM}}$ as 
\begin{align*}
\log\mathcal{U}_n^{\mathrm{IM}}=\displaystyle\sum\limits_{k=0}^n{\log(1-a_k+\alpha_1a_k)+\displaystyle\sum\limits_{k=0}^n{\log(1-b_k+\alpha_2b_k)}}.
\end{align*}
Thanks to \eqref{ln-ineq}, we infer that
\begin{align}\label{est-U-IM-A-01}
-(1-\alpha_1)\displaystyle\sum\limits_{k=0}^n{\dfrac{a_k}{1-a_k+\alpha_1a_k}}-(1-\alpha_2)\displaystyle\sum\limits_{k=0}^n{\dfrac{b_k}{1-b_k+\alpha_2b_k}} & \le \log\mathcal{U}_n^{\mathrm{IM}}, 
\end{align}
and 
\begin{align}\label{est-U-IM-A-02}
 \log\mathcal{U}_n^{\mathrm{IM}} &\le   -(1-\alpha_1)\displaystyle\sum\limits_{k=0}^n{a_k}  - (1-\alpha_2)\displaystyle\sum\limits_{k=0}^n{b_k}.
\end{align}
Invoking Lemma \ref{lem:seri}, it gives us
\begin{align*}
\sum\limits_{k=0}^{\infty}{\dfrac{a_k}{1-a_k+\alpha_1a_k}} \sim \displaystyle\sum\limits_{k=0}^{\infty}{a_k}, \quad \sum\limits_{k=0}^{\infty}{\dfrac{b_k}{1-b_k+\alpha_2b_k}} \sim \displaystyle\sum\limits_{k=0}^{\infty}{b_k},
\end{align*}
and yields that \eqref{U-IM-A-1st} by combining with \eqref{est-U-IM-A-01} and \eqref{est-U-IM-A-02}. The proof is complete.
\end{proof}

\begin{theorem}\label{theo-IM-B}
Let $\alpha_1, \alpha_2 \in [0,1]$ such that $0<\alpha_1+\alpha_2<2$ and $(a_n)_{n \in \mathbb{N}}$, $(b_n)_{n \in \mathbb{N}}$ be two sequences in the interval $[0,1]$. Assume further that
\begin{align}\label{A-IM-B}
1-a_k - \alpha_1 a_k>0 \ \mbox{ và } \ 1-b_k - \alpha_2 b_k>0, \quad \mbox{ for all } k \in \mathbb{N}.
\end{align}
Then, the OLEB sequence $(\mathcal{L}_n^{\mathrm{I}})_{n\in\mathbb{N}}$ of iterative process~\eqref{xn-IM} is defined by
\begin{align}\label{L-IM-B}
& \mathcal{L}_n^{\mathrm{IM}} = \prod_{k=0}^n (1-a_k - \alpha_1 a_k)(1-b_k - \alpha_2 b_k), \quad  n \in \mathbb{N}.
\end{align}
Moreover, the following statement is true
\begin{align}\label{L-IM-B1}
    \lim\limits_{n\to \infty}\mathcal{L}^{\mathrm{IM}}_n=0\Longleftrightarrow \displaystyle\sum_{k=0}^{\infty}(a_k+b_k)=+\infty.  
\end{align}
\end{theorem}
\begin{proof}
For the iterative sequence $(x_n)_{n \in \mathbb{N}}$ defined by \eqref{xn-IM-yn}, one can bound the distance between $y_n$ and $x^*$ as below
\begin{align}
\|y_n - x^*\| & = \|(1-a_n)x_n + a_nT_1(x_n) - (1-a_n)x^* - a_nT_1(x^*)\| \notag\\
& \ge (1-a_n)\|x_n - x^*\| - a_n\|T_1(x_n) - T_1(x^*)\| \notag\\
& \ge (1-a_n-\alpha_1 a_n)\|x_n-x^*\|, \quad n\in \mathbb{N}.\notag
\end{align}
On the other hand, one has
\begin{align}\notag
\|x_{n+1} - x^*\| & \ge (1-b_n-\alpha_2 b_n)\|y_n-x^*\|, \quad n\in \mathbb{N}.
\end{align}
Thus, combining them together yields that $(\mathcal{L}_n^{\mathrm{IM}})_{n \in \mathbb{N}}$ defined by \eqref{L-IM-B} is a sequence of lower error bounds of scheme~\eqref{xn-IM}. Furthermore, we are able to prove that this lower bound is optimal by choosing $x^* \in \Omega$ and two non-expansive mappings $T_1,T_2$ given by 
$$T_1(x)=-\alpha_1x+(1+\alpha_1)x^*, \quad T_2(x)=-\alpha_2x + (1+\alpha_2)x^*, \quad x\in \Omega.$$
Finally, to show \eqref{L-IM-B1}, let us first apply \eqref{ln-ineq} to arrive at
\begin{align}\label{L-IM-B3}
 -\displaystyle\sum_{k=0}^{n}\left(\dfrac{\sigma_1 a_k}{1-\sigma_1 a_k}+\dfrac{\sigma_2 b_k}{1-\sigma_2 b_k}\right)\le \log \mathcal{L}^{\mathrm{IM}}_n \le - \displaystyle\sum_{k=0}^{n}\left(\sigma_1 a_k+\sigma_2 b_k\right),
\end{align}
where $\sigma_1=\alpha_1+1$ and $\sigma_2=\alpha_2+1$. Thanks to Lemma \ref{lem:seri} once again with noticing that $\sigma_1, \sigma_2 \ge 1$, one has
\begin{align*}
 \sum_{k=0}^{\infty}\left(\dfrac{\sigma_1 a_k}{1-\sigma_1 a_k}+\dfrac{\sigma_2 b_k}{1-\sigma_2 b_k}\right) \sim \sum_{k=0}^{\infty}\left(\sigma_1 a_k + \sigma_2 b_k\right) \sim \sum_{k=0}^{\infty}\left( a_k+ b_k\right).
\end{align*}
This allows us to conclude \eqref{L-IM-B1} from \eqref{L-IM-B3} and complete the proof.
\end{proof}

\section{Convergence theorems}
\label{sec:main}

This section aims to state and prove convergence theorems concerning the optimal error bounds. By controlling the sequence of OEBs, we derive some necessary and sufficient conditions for the convergence of the iterative process. This makes it very interesting to study the OEB sequence in various iterative procedures. Although in this paper, we only illustrate the study of OEBs for Ishikawa and the modified Ishikawa process, it is worth noting that our novelty with OEBs can be applied to various iterative processes for approximating common fixed points of non-expansive mappings. 

\subsection{Convergence rate estimates from OEBs}

The following theorem finds two interesting assumptions that are sufficient to obtain the convergence of Ishikawa iterative process~\eqref{xn-I}. The result can be simply obtained from Theorem \ref{theo-IA} by sending the OUEB sequence to zero. Further, we also give a necessary and sufficient condition for the convergence by controlling both optimal lower and upper error bounds. Also, from the viewpoint of numerical analysis, determining the convergence rate plays a central role in iterative methods. We shall present the convergence analysis for two iterative schemes of Ishikawa type via the sequence of OEBs.

\begin{theorem}\label{coro-I}
Let $\Omega$ be a convex subset in a Banach space $\mathbb{X}$ and $(a_n)_{n \in \mathbb{N}}$, $(b_n)_{n \in \mathbb{N}}$ be two sequences in the interval $[0,1]$. Assume further that $T_1 \in \mathbf{Ne}_{\alpha_1}(\Omega)$, $T_2 \in \mathbf{Ne}_{\alpha_2}(\Omega)$ have a common fixed point $x^* \in \mathcal{F}(T_1, T_2)$ for some $\alpha_1, \alpha_2 \in [0,1]$ satisfying $0<\alpha_1+\alpha_2<2$. Let $(x_n)_{n \in \mathbb{N}}$ be an iterative sequence defined by~\eqref{xn-I}. Then, $(x_n)$ converges to $x^*$ as $n$ tends to infinity if one of the two following assumptions is satisfied:
\begin{enumerate}
\item [(i)] $\alpha_1 \in [0,1]$,  $\alpha_2<1$ and $\displaystyle{\sum_{k=0}^{\infty} b_k = +\infty}$;
\item [(ii)] $\alpha_1 \in [0,1)$, $\alpha_2=1$ and $\displaystyle{\sum_{k=0}^{\infty} a_kb_k = +\infty}$. 
\end{enumerate}
In particular, if $\alpha_2<1$ and \eqref{A-theo-IB} holds, then one has
\begin{align}\label{R-coro-I}
\lim_{n \to \infty} x_n = x^* \Longleftrightarrow \displaystyle{\sum_{k=0}^{\infty} b_k = +\infty}.
\end{align}
Moreover, when $\alpha_2<1$ and $\displaystyle\sum_{k=0}^{\infty} b_k = +\infty$, we then have 
\begin{align}\label{Rate-I}
 \log \mathcal{R}^{-1}(x_{n+1},x_0,x^*) \approx \sum_{k=0}^n b_k,
\end{align}
if there exists $\varepsilon \in (0,1)$ such that
\begin{align}\label{cond-Ra-1}
b_k \le \frac{1-\varepsilon}{1+\alpha_2(1-a_k+\alpha_1 a_k)}, \ \mbox{ for every } k \in \mathbb{N}.
\end{align}
\end{theorem}
\begin{proof}
If one of the two following assumptions in $(i)$ and $(ii)$ satisfies, then Theorem~\ref{theo-IA} gives us the limit of the OUEB sequence $(\mathcal{U}_n^{\mathrm{I}})_{n\in\mathbb{N}}$ defined by \eqref{U-IA} is zero, i.e. $\displaystyle{\lim_{n \to \infty}}  \mathcal{U}_n^{\mathrm{I}}=0$ and hence, we arrive at $\displaystyle{\lim_{n \to \infty}}  x_n = x^*$. Next, if we assume that $\alpha_2<1$ and the assumption~\eqref{A-theo-IB} given, then the following statement holds
\begin{align}\notag
\sum_{k=0}^{\infty} b_k = +\infty \Longleftrightarrow \lim_{n \to \infty}\mathcal{L}_n^{\mathrm{I}} = \lim_{n \to \infty}\mathcal{U}_n^{\mathrm{I}} = 0,
\end{align}
by combining Theorem \ref{theo-IA} and Theorem \ref{theo-IB}. And we finally obtain the desired result in \eqref{R-coro-I}. Moreover, Theorem \ref{theo-IA} and \ref{theo-IB} also give us the following estimate
\begin{align}\label{Ra-LU}
\mathcal{L}^{\mathrm{I}}_n \le \mathcal{R}(x_{n+1},x_0,x^*) \le \mathcal{U}^{\mathrm{I}}_n, \ \mbox{ for every } n \in \mathbb{N},
\end{align}
where $(\mathcal{U}^{\mathrm{I}}_n)$ and $(\mathcal{L}^{\mathrm{I}}_n)$ are OUEB and OLEB sequence given as in~\eqref{U-IA} and \eqref{L-IB}, respectively. Here, it is worth noting that \eqref{A-theo-IB} is valid and $\mathcal{L}^{\mathrm{I}}_n$ is well-defined if there exists a real number $\delta \in \left(0,\frac{1}{1+\alpha_2}\right)$ such that $(b_n) \subset [0,\delta]$ or there exists $\varepsilon \in (0,1)$ satisfying \eqref{cond-Ra-1}.  Combining \eqref{Ra-LU} with two estimates in \eqref{est-U-IA} and \eqref{est-L-IB-01}, one obtains that
\begin{align}\label{Ra-LU-2}
(1-\alpha_2)\sum_{k=0}^n b_k \le \log \mathcal{R}^{-1}(x_{n+1},x_0,x^*) \le (1+\alpha_2) \sum_{k=0}^n \frac{b_k}{1-\big(1+\alpha_2 (1-a_k+\alpha_1a_k)\big)b_k}.
\end{align}
It is clear to see that $\alpha_1 \le 1-a_k+\alpha_1a_k \le 1$ and therefore, if \eqref{cond-Ra-1} satisfies for some $\varepsilon \in (0,1)$, then we also obtain that
\begin{align}\notag 
(1-\alpha_2)\sum_{k=0}^n b_k \le \log \mathcal{R}^{-1}(x_{n+1},x_0,x^*) \le \frac{1+\alpha_2}{\varepsilon} \sum_{k=0}^n b_k.
\end{align}
The proof of \eqref{Rate-I} is complete from the definition of two equivalent sequences in \eqref{def-approx}.
\end{proof}

\begin{remark}\label{Rmk-theo-1}
The conclusion of \eqref{Rate-I} still holds if we replace \eqref{cond-Ra-1} by the following assumption
\begin{align*}
(b_n) \subset [0,\delta], \mbox{ for some } \delta \in \left(0,\frac{1}{1+\alpha_2}\right)
\end{align*}
Indeed, if $b_k \le \delta$ for some $\delta \in \left(0,\frac{1}{1+\alpha_2}\right)$, then \eqref{Ra-LU-2} implies that
\begin{align}\notag 
(1-\alpha_2)\sum_{k=0}^n b_k \le \log \mathcal{R}^{-1}(x_{n+1},x_0,x^*) \le \frac{1+\alpha_2}{1-(1+\alpha_2)\delta} \sum_{k=0}^n b_k,
\end{align}
which allows us to conclude \eqref{Rate-I}. 
\end{remark}

From two previous theorems: Theorem \ref{theo-IM-A} and Theorem \ref{theo-IM-B}, we directly obtain the following theorem that concerning the convergence analysis of the modified Ishikawa iterative process~\eqref{xn-IM}.

\begin{theorem}\label{coro-IM}
Let $\Omega$ be a convex subset in a Banach space $\mathbb{X}$ and $(a_n)_{n \in \mathbb{N}}$, $(b_n)_{n \in \mathbb{N}}$ be two sequences in the interval $[0,1]$. Assume that $T_1 \in \mathbf{Ne}_{\alpha_1}(\Omega)$, $T_2 \in \mathbf{Ne}_{\alpha_2}(\Omega)$ have a common fixed point $x^* \in \mathcal{F}(T_1, T_2)$ for some $\alpha_1, \alpha_2 \in (0,1)$. Further, given $(x_n)$ be an iterative sequence defined by~\eqref{xn-IM}. Then, there holds
\begin{align}\label{coro-IM-1}
\sum_{k=0}^{\infty}(a_k+b_k)=+\infty \Longrightarrow \lim_{n \to \infty} x_n = x^*.
\end{align}
In particular, under an additional assumption ~\eqref{A-IM-B}, the following statement holds true 
\begin{align}\label{coro-IM-2}
\sum_{k=0}^{\infty}(a_k+b_k)=+\infty \Longleftrightarrow \lim_{n \to \infty} x_n = x^*.
\end{align}
Moreover, if there exists $\varepsilon \in (0,1)$ such that 
\begin{align}\label{cond-Ra-IM}
(a_n) \subset \left[0,\frac{1-\varepsilon}{1+\alpha_1}\right], \mbox{ and } \ (b_n) \subset \left[0,\frac{1-\varepsilon}{1+\alpha_2}\right], 
\end{align}
then we have 
\begin{align}\label{Rate-IM}
 \log \mathcal{R}^{-1}(x_{n+1},x_0,x^*) \approx \sum_{k=0}^n (a_k+b_k).
\end{align}
\end{theorem}
\begin{proof}
Applying a similar argument as in the previous proofs of Theorem \ref{coro-I}, we are able to conclude \eqref{coro-IM-1} and \eqref{coro-IM-2}. 
Moreover, it also gives us 
\begin{align}\label{Ra-LU-IM}
\mathcal{L}^{\mathrm{IM}}_n \le \mathcal{R}(x_{n+1},x_0,x^*) \le \mathcal{U}^{\mathrm{IM}}_n, \ \mbox{ for every } n \in\mathbb{N},
\end{align}
where $(\mathcal{U}^{\mathrm{IM}}_n)$ and $(\mathcal{L}^{\mathrm{IM}}_n)$ are the sequence of optimal upper and lower error bounds in~\eqref{U-IM-A} and \eqref{L-IM-B}, respectively. Combining \eqref{Ra-LU-IM} with \eqref{est-U-IM-A-02} and \eqref{L-IM-B3}, one gets that
\begin{align}\label{Ra-LU-IM-2}
\min\{1-\alpha_1;1-\alpha_2\}\sum_{k=0}^n (a_k+b_k) \le \log \mathcal{R}^{-1}(x_{n+1},x_0,x^*),
\end{align}
and 
\begin{align}\notag
\log \mathcal{R}^{-1}(x_{n+1},x_0,x^*) \le \sum_{k=0}^{n}\left[\frac{(1+\alpha_1) a_k}{1-(1+\alpha_1) a_k}+\frac{(1+\alpha_2) b_k}{1-(1+\alpha_2) b_k}\right].
\end{align}
Taking \eqref{cond-Ra-IM} into account, we arrive at
\begin{align}\label{Ra-LU-IM-3}
\log \mathcal{R}^{-1}(x_{n+1},x_0,x^*) \le \frac{1}{\varepsilon} \max\{1+\alpha_1;1+\alpha_2\} \sum_{k=0}^{n} (a_k+b_k).
\end{align}
At this stage, it concludes \eqref{Rate-IM} by making use of \eqref{Ra-LU-IM-2} and \eqref{Ra-LU-IM-3}.
\end{proof}

\subsection{Comparison strategy with OEBs}

Building on the idea of the OEB sequence, we shall derive a comparison on convergence rate between two iterative processes mentioned in previous sections.

\begin{theorem}\label{theo:I-IM}
Let $\Omega$ be a convex subset in a Banach space $\mathbb{X}$ and two sequences $(a_n)_{n\in\mathbb{N}}$ and $(b_n)_{n\in\mathbb{N}}$ in $[0,1]$.  Assume that $T_1 \in \mathbf{Ne}_{\alpha_1}(\Omega)$, $T_2 \in \mathbf{Ne}_{\alpha_2}(\Omega)$ have a common fixed point $x^* \in \mathcal{F}(T_1, T_2)$ for some $\alpha_1,\alpha_2\in (0,1)$. Further, given $(x^{\mathrm{I}}_n)$ and $(x^{\mathrm{IM}}_n)$ be two iterative sequences defined by~\eqref{xn-I} and~\eqref{xn-IM} respectively, with the same initial point $x_0 \in \Omega$. Then, the following statement holds true 
\begin{align}\label{cond:I-IM-b}
\lim_{n \to \infty} x^{\mathrm{I}}_n=\lim_{n \to \infty} x^{\mathrm{IM}}_n=x^* \Longleftrightarrow \sum_{k=0}^{\infty}b_k=+\infty.
\end{align}
Moreover, the sequence $(x^{\mathrm{IM}}_n)$ converges to $x^*$ faster than $(x^{\mathrm{I}}_n)$ under following additional conditions 
\begin{align}\label{cond:I-IM-a}
\sum_{k=0}^{\infty}b_k =+\infty, \mbox{ and } b_k \le \frac{(1-\alpha_1)a_k}{1+\alpha_2(1-a_k+\alpha_1 a_k)} \mbox{ for every } k \in \mathbb{N}.
\end{align}
\end{theorem}
\begin{proof}
Let us prove the first statement in \eqref{cond:I-IM-b} of this theorem. Firstly, we remark that
\begin{align}\label{IMfasterI}
\displaystyle{\lim_{n \to \infty} x^{\mathrm{I}}_n=x^*} \Longrightarrow \lim_{n \to \infty} x^{\mathrm{IM}}_n=x^*.
\end{align}
Indeed, if  $\displaystyle{\lim_{n \to \infty} x^{\mathrm{I}}_n=x^*}$, then thanks to Theorem \ref{coro-I}, it obtains $\displaystyle\sum_{k=0}^{\infty}b_k=+\infty$. From Theorem \ref{theo-IM-A}, we infer that  $\displaystyle{\lim_{n \to \infty} x^{\mathrm{IM}}_n=x^*}$. On the other hand, under assumption $\displaystyle\sum_{k=0}^{\infty}b_k=+\infty$, Theorem \ref{coro-I} and Theorem \ref{coro-IM} allow us to conclude 
$$\lim_{n \to \infty} x^{\mathrm{I}}_n=\lim_{n \to \infty} x^{\mathrm{IM}}_n=x^*.$$

Now, we assume that \eqref{cond:I-IM-a} is provided. Since $\displaystyle\sum_{k=0}^{\infty}b_k=+\infty$, we are allowed to obtain that both sequences $(x^{\mathrm{I}}_n)$ and $(x^{\mathrm{IM}}_n)$ converge to $x^*$. In addition, the second condition in \eqref{cond:I-IM-a} implies that
\begin{align}\notag 
b_k \le \frac{(1-\alpha_1)a_k}{1+\alpha_2(1-a_k+\alpha_1 a_k)} < \frac{1}{1+\alpha_2(1-a_k+\alpha_1 a_k)} \quad \mbox{for every } k \in \mathbb{N},
\end{align}
which ensures the existence of the sequence of optimal lower error bound $\mathcal{L}^{\mathrm{I}}_n$ for iteration \eqref{xn-I}. Moreover, we also have
\begin{align}\label{est-24}
\mathcal{R}\big(x^{\mathrm{IM}}_n,x^{\mathrm{I}}_n,x^*\big) \leq \frac{\mathcal{U}^{\mathrm{IM}}_n}{\mathcal{L}^{\mathrm{I}}_n}=\displaystyle\prod_{k=0}^{n} \frac{(1-a_k + \alpha_1 a_k)(1-b_k + \alpha_2 b_k)}{1- b_k - \alpha_2b_k(1 - a_k + \alpha_1a_k)},
 \end{align}
 where the ratio $\mathcal{R}\big(x^{\mathrm{IM}}_n,x^{\mathrm{I}}_n,x^*\big)$ is defined as in \eqref{Ra}. In order to estimate the right-hand side of \eqref{est-24}, it is worth noting that \eqref{cond:I-IM-a} implies to
 \begin{align*}
 \frac{1-a_k + \alpha_1 a_k}{1- b_k - \alpha_2b_k(1 - a_k + \alpha_1a_k)} \le 1.
 \end{align*}
 Substituting this inequality into \eqref{est-24}, it follows that
 \begin{align}\notag 
\mathcal{R}\big(x^{\mathrm{IM}}_n,x^{\mathrm{I}}_n,x^*\big)  \le \displaystyle\prod_{k=0}^{n} (1-b_k + \alpha_2 b_k).
 \end{align}
In a completely similar way in the proof of Theorem \ref{theo-IA}, it enables us to write
\begin{align}\notag 
 \lim\limits_{n\to \infty}\displaystyle\prod_{k=0}^{n}(1-b_k+\alpha_2 b_k)=0 \Longleftrightarrow \displaystyle\sum_{k=0}^{\infty}b_k=+\infty.
\end{align}
Hence, under assumption $\displaystyle\sum_{k=0}^{\infty}b_k=+\infty$ from \eqref{cond:I-IM-a}, we may conclude that 
\begin{align*}
\lim_{n \to \infty}\mathcal{R}\big(x^{\mathrm{IM}}_n,x^{\mathrm{I}}_n,x^*\big) = 0,
\end{align*}
and this means that $x^{\mathrm{IM}}_n$ converges to $x^*$ faster than $x^{\mathrm{I}}_n$.
\end{proof}

\section{Numerical experiments}
\label{sec:tests}

In this section, some numerical examples will be carried out to validate the theoretical results in previous sections. Besides, the analysis of convergence and comparison strategy based on the idea of OEBs are also proposed to illustrate the theoretical proofs in Theorem~\ref{coro-I},~\ref{coro-IM} and Theorem~\ref{theo:I-IM}. First, let us consider $\mathbb{X} = \mathbb{R}$ and two constants $\alpha_1, \alpha_2 \in [0,1]$ satisfying $0<\alpha_1 + \alpha_2 <2$. In all the experiments, we consider the domain $\Omega:= \left[\frac{1}{4},3\right] \subset \mathbb{X}$ and two following mappings 
\begin{align}\label{T1T2}
T_1(x) = \sqrt{\alpha_1 x + 1-\alpha_1}; \quad T_2(x) = \alpha_2 \sin(x-1) + 1, \quad x \in \Omega.
\end{align}
Then, it is easy to verify that $T_1 \in \mathbf{Ne}_{\alpha_1}(\Omega)$, $T_2 \in \mathbf{Ne}_{\alpha_2}(\Omega)$ and $\mathcal{F}(T_1,T_2) = \{x^*\}$ with the common fixed point $x^*=1$. Here, we perform the numerical tests of two Ishikawa-type iterative sequences~\eqref{xn-I} and~\eqref{xn-IM-yn} that are associated with $T_1$, $T_2$ and the initial value $x_0=2$. 

Further, at iteration $n$ of each scheme, the convergence error can be defined as 
\begin{align}\label{def:err}
\mathrm{Err}_n := \mathcal{R}(x_{n},x_0,x^*), \quad n \in \mathbb{N}.
\end{align}

In each subsection as below, we restrict ourselves to the numerical tests for the Ishikawa iterative scheme~\eqref{xn-I} and the modified Ishikawa scheme~\eqref{xn-IM-yn} to a common fixed point of two mappings $T_1$, $T_2$, respectively. In the following experiments, we numerically verified the theoretical results of previous theorems based on the idea of OEBs. Numerical experiments are carried out to illustrate the relevance of theoretical analysis proved in Theorem~\ref{coro-I},~\ref{coro-IM} and~\ref{theo:I-IM}.  In addition, we also perform some tests to estimate the rate of convergence of these methods via the OEBs. 

\subsection{Ishikawa iterative process}
\label{sec:ishikawa_tests}
Let us first recall the Ishikawa scheme~\eqref{xn-I} that can be written in a two-step process
\begin{align}\label{test-I}
\begin{cases} 
x_0 \in \Omega,\\
y_n = (1-a_n)x_n + a_nT_1(x_n), \\
x_{n+1} = (1-b_n)x_n + b_nT_2(y_n), \quad n \in \mathbb{N}.\end{cases}
\end{align}

\subsubsection{Convergence results}

For the first test, we verify the theoretical convergence results of the sequence $(x_n)$ to fixed point $x^*$ in Theorem~\ref{coro-I}. It is worth mentioning that Theorem~\ref{coro-I} plays a key role in choosing prescribed parameters to obtain a good behavior of convergence of our iterative scheme.

Here, we perform several tests which correspond to both cases with $\alpha_2<1$ and $\alpha_2 = 1$. In order test the first case, let us fix 
\begin{align}\label{alpha12}
\alpha_1 = \frac{1}{2} \ \mbox{ and } \ \alpha_2=\frac{1}{5}.
\end{align}
In this case, as proved in Theorem~\ref{coro-I}$(i)$, the necessary and sufficient conditions for the sequence $(x_n)$ to converge to $x^*$ is $\sum{b_n}=+\infty$. Regarding this test, we choose $(a_n)$ as a sequence of random points selected in $[0,1]$, and our theoretical results will be illustrated with several examples of $\sum{b_n}$, where
\begin{align}\label{eq:bn} 
\begin{split}
& \text{Test}\ \ (1): \ b_n = \mathrm{rand}([0,1]) \approx 1;\\
& \text{Test}\ \ (2): \ b_n = \frac{\sqrt{n+1}+\sin(n+1)}{2\sqrt{n+1}+3} \approx \frac{1}{2}; \\
& \text{Test}\ \ (3): \ b_n = \frac{n+2}{3(n+1)\sqrt[3]{n+1} + 4} \approx \frac{1}{\sqrt[3]{n+1}}; \\
& \text{Test}\ \ (4): \ b_n = \frac{2n+3}{3(n+1)^3+1} \approx \frac{1}{(n+1)^2}, 
\end{split}
\end{align}
respectively. Here, with a slight abuse of notation, we use $\mathrm{rand}([0,1])$ stands for a random number belonging in $[0,1]$.

Figure~\ref{fig:1} shows the plot of the logarithm of error against each iteration $n$ as aforementioned in~\eqref{def:err}. The logarithmic-scale on the vertical axis and the linear scale on the horizontal axis for the maximal number of iterations is $N=50$ and $N=500$, respectively. We perform four tests corresponding to series $\sum{b_n}$ chosen in~\eqref{eq:bn} with red (with pluses), green (with circle markers), dashed blue, and dotted black lines. From both Figure~\ref{fig:1}(a) and~\ref{fig:1}(b), it can be seen that: first, the faster the convergence of series $\sum{b_n}$ is, the faster the convergence rate of our iterative process will be. The second observation is that when $\sum b_n < +\infty$ (the last case, shown in the dotted black line), the iteration will fail to converge. In other words, it does not yield convergence of the iterative sequence to a common fixed point as expected. To be more precise, we also increase the number of iterations up to $N=500$ and plot results in Figure~\ref{fig:1}(b), it also gives us a similar conclusion. This observation completely validates what we proved in the theoretical findings in Theorem~\ref{coro-I}.

\begin{figure}[H]
\centering
\subfigure[Four test models in \eqref{eq:bn} with $N=50$ iterations.]
{\includegraphics[width=0.49\linewidth]{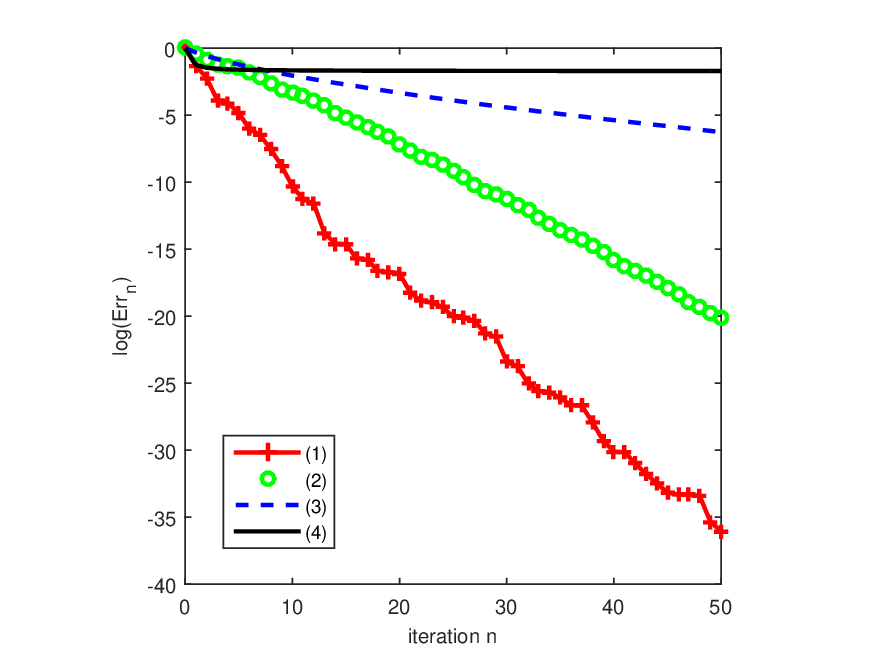}}
\subfigure[The two last test models in \eqref{eq:bn} with $N=500$ iterations.]{\includegraphics[width=0.49\linewidth]{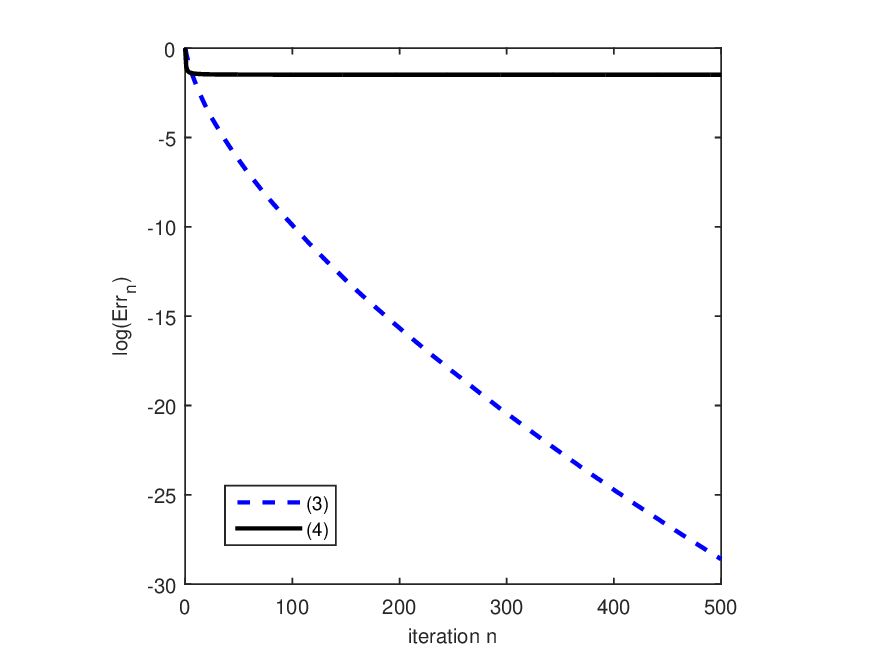}}
\caption{Errors plotted against $n$ on a log-scale with four test models of the series $\sum{b_n}$ in~\eqref{eq:bn}, where $\alpha_1 = \frac{1}{2}$ and $\alpha_2 = \frac{1}{5}$.}
\label{fig:1}
\end{figure}  

Moreover, the question may arise here: what happens if we are considering the series $\sum{a_n}<+\infty$ or $\sum{a_n}=\infty$? - Here, the advantage is that the choice of  $\sum{a_n}$ has no effect on the convergence process of Ishikawa sequence~\eqref{xn-I}. To clarify the answer, we plot the behavior of errors with two sequences of $b_n$ such that $\sum b_n = \infty$ and $\sum b_n < \infty$, respectively. For example, let us choose
\begin{align}\label{bn-fig1b}
b_n = \frac{2(n+1)^2+1}{3(n+1)^2\sqrt{n+1}+4} \approx \frac{1}{\sqrt{n+1}}; \ \mbox{ or } \  b_n = \frac{n+4}{4(n+1)^3+5} \approx \frac{1}{(n+1)^2}.
\end{align}
In particular, for each case of $(b_n)_{n \in \mathbb{N}}$ in~\eqref{bn-fig1b}, we choose three sequences $(a_n)_{n \in \mathbb{N}}$ as below
\begin{align}\label{an-fig1b}
\begin{split}
& \text{Test} \ \ (1): \ a_n = \mathrm{rand}([0,1]) \approx 1;  \\
& \text{Test} \ \ (2): \ a_n = \frac{\sin^2(n+1)}{\sqrt{(n+1)^2+10}} \approx \frac{1}{\sqrt{n+1}};   \\
& \text{Test} \ \ (3): \ a_n = \frac{2\exp\left(-\frac{1}{n}+1\right)+1}{4(n+1)^3|\sin(n+1)|+2\exp\left(-\frac{1}{n}+1\right)} \approx \frac{1}{(n+1)^3},
\end{split}
\end{align}
and respectively indicated by dashed-dotted red, dotted blue, and dashed black lines. Figure~\ref{fig:1b} confirms our comments: the property of $\sum{a_n}$ does not affect the convergence results in~$(i)$ of Theorem~\ref{coro-I}.

\begin{figure}[H]
\centering
\subfigure[With $b_n = \displaystyle{\frac{2(n+1)^2+1}{3(n+1)^2\sqrt{n+1}+4}}$.]{\includegraphics[width=0.49\linewidth]{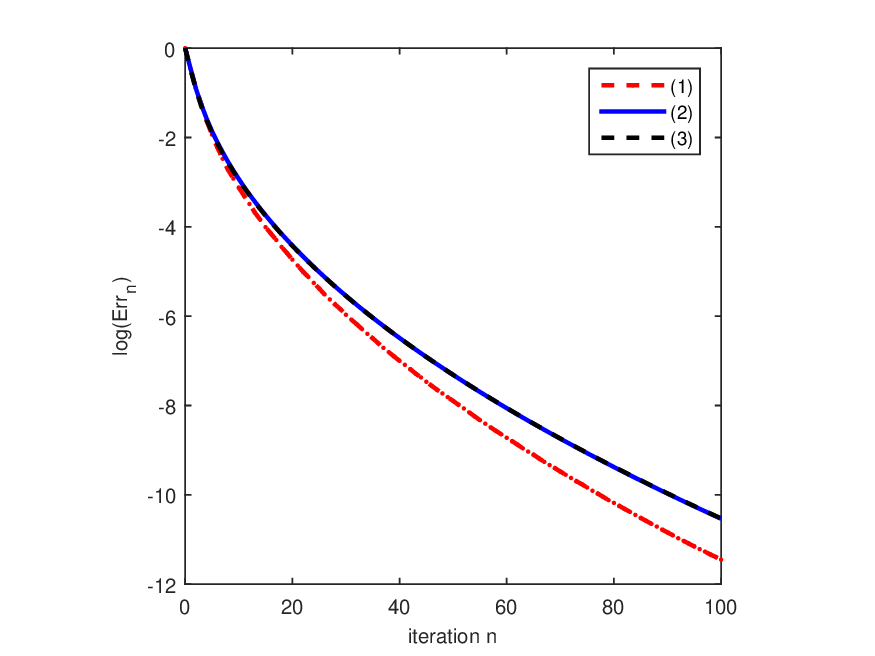}}
\subfigure[With $b_n = \displaystyle{\frac{n+4}{4(n+1)^3+5}}$.]{\includegraphics[width=0.49\linewidth]{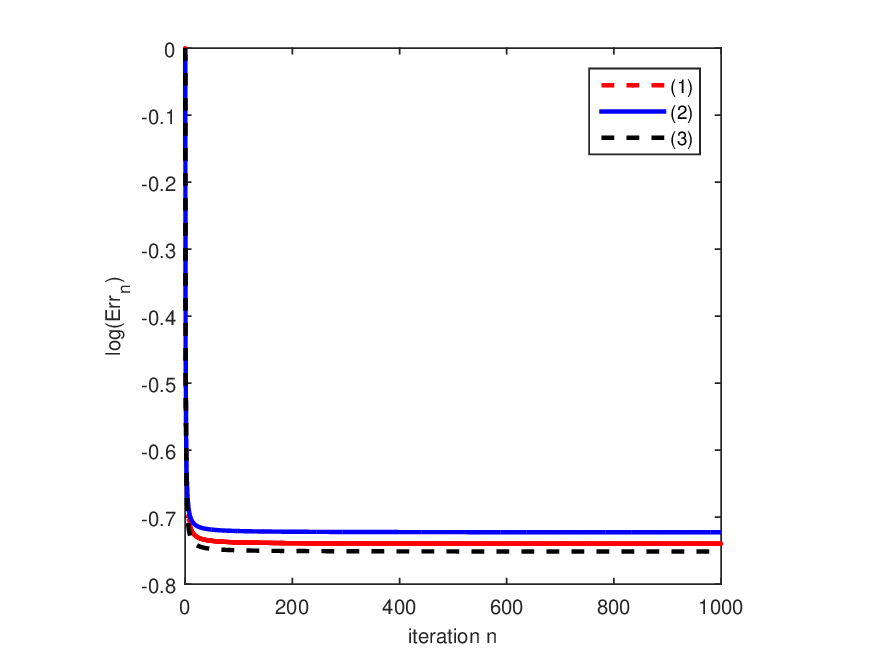}}
\caption{Errors plotted against $n$ on a log-scale with three test models of the series $\sum{a_n}$ in~\eqref{an-fig1b}, where $\alpha_1 = \frac{1}{2}$ and $\alpha_2 = \frac{1}{5}$.}
\label{fig:1b}
\end{figure}

Next, the validity of the second convergence statement in Theorem~\ref{coro-I} under assumption $(ii)$ will be performed via some numerical tests. Let us fix 
$$\alpha_1 = \frac{1}{2} \ \mbox{ and } \ \alpha_2=1.$$ 
In Figure~\ref{fig:2}(a) and~\ref{fig:2}(b), we execute four test models as in~\eqref{eq:bn} only with $a_n = \mathrm{rand}([0,1])$ the sequence chosen randomly between 0 and 1. These figures yield convergence of the sequence~\eqref{xn-I} to a common fixed point $x^*$ if $\sum a_n b_n = \infty$. Further, it notices that we arrive at the same conclusion as that of the previous examples with the statement $(i)$ of Theorem~\ref{coro-I}. 

\begin{figure}[H]
\centering
\subfigure[Four test models in \eqref{eq:bn} with $N=50$ iterations.]{\includegraphics[width=0.49\linewidth]{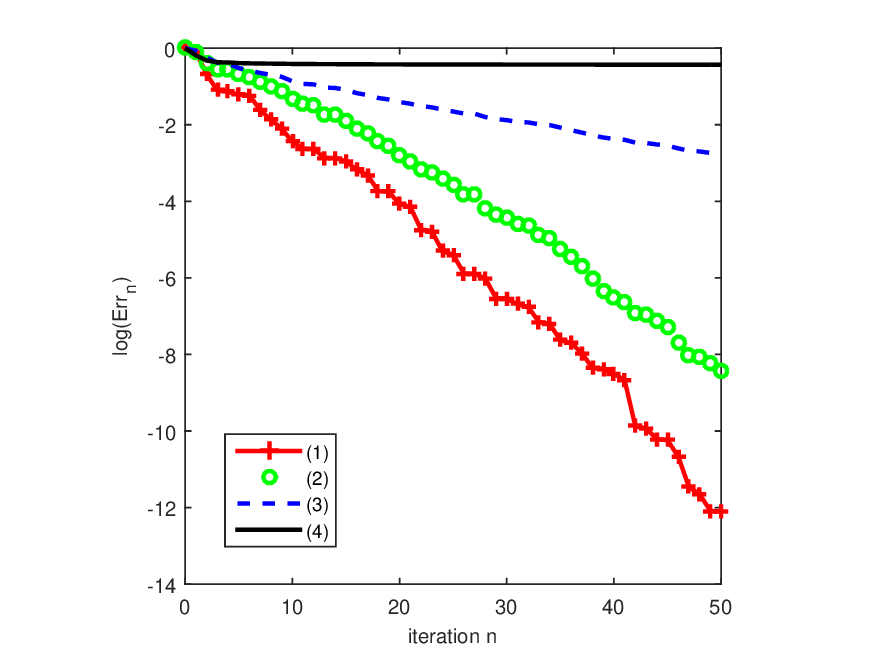}}
\subfigure[The two last test models in \eqref{eq:bn} with $N=500$ iterations.]{\includegraphics[width=0.49\linewidth]{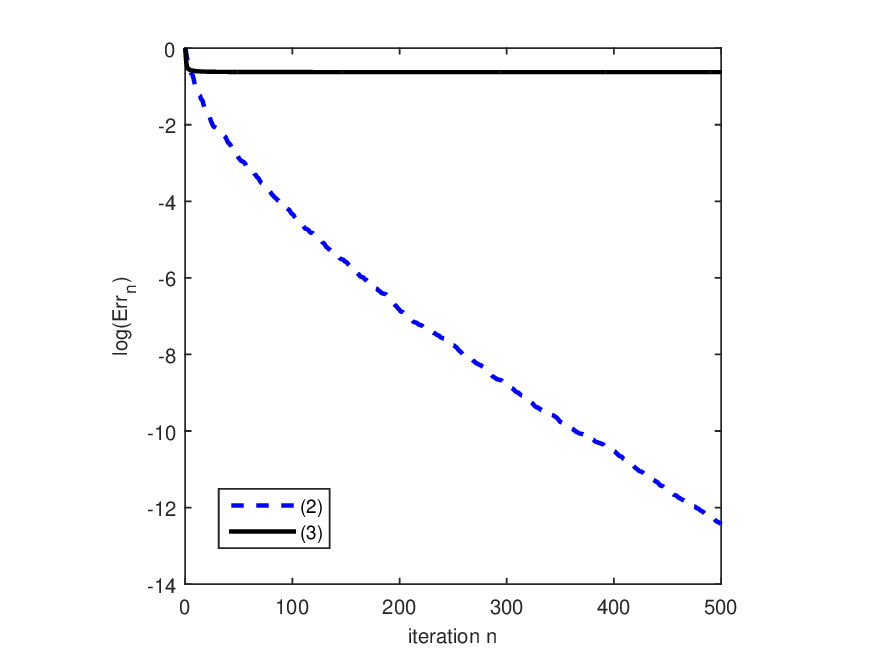}}
\caption{Errors plotted against $n$ on a log-scale with four test models of the series $\sum{b_n}$ in~\eqref{eq:bn}, where $\alpha_1 = \frac{1}{2}$ and $\alpha_2 = 1$.}
\label{fig:2}
\end{figure}

In the next test, we verify the convergence results for different choices of data, where the sum $\sum{a_nb_n}$ is finite or infinite. For example, it enables us to choose
\begin{align}\label{eq:anbn}
\begin{split}
& \text{Test} \ \ (1): \ a_n = \frac{1}{n+1},  \quad b_n = \mathrm{rand}([0,1]);  \\
& \text{Test} \ \ (2): \ a_n = \frac{1}{(n+1)^{\frac{2}{3}}},  \quad b_n = \frac{1}{(n+1)^{\frac{3}{4}}};  \\
& \text{Test} \ \ (3): \ a_n = \frac{1}{n+1},  \quad b_n = \frac{\mathrm{rand}([0,1])}{n+1};\\
& \text{Test} \ \ (4): \ a_n = \frac{1}{\sqrt{n+1}},  \quad b_n = \frac{1}{(n+1)^2},
\end{split}
\end{align}
to verify the second statement $(ii)$, respectively. 

Figure~\ref{fig:3} presents errors plotted against the number of iterations $n$ on a log-scale for these two models whenever $a_n$ and $b_n$ are chosen as above. The dashed red, dotted green, blue and black lines in the figure show the behavior of iteration errors committed by these four test cases (1), (2), (3) and (4) in~\eqref{eq:anbn}, respectively. It can be observed from the results that the iteration error decreases steadily whenever $\sum a_n b_n = \infty$ in the first case (the dashed red line), otherwise $\sum a_n b_n < \infty$ in three other cases, the resulting iterative Ishikawa method~\eqref{xn-I} may not converge. 

\begin{figure}[H]
\centering
{\includegraphics[width=0.65\linewidth]{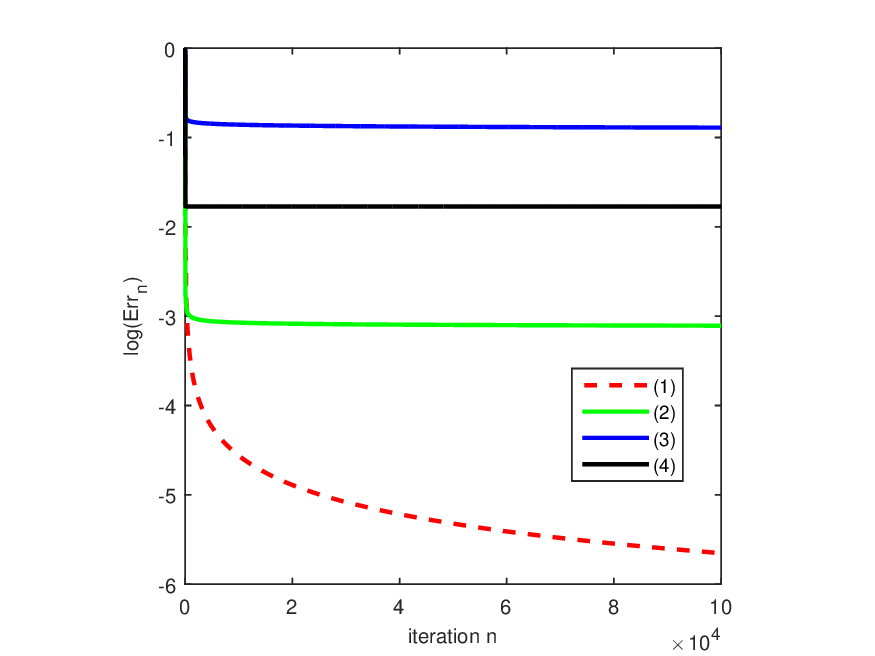}}
  \caption{Errors plotted against $n$ on a log-scale for various choices of $a_n$ and $b_n$ as in~\eqref{eq:anbn}.}
\label{fig:3}
\end{figure}

At last, the reader may wish to look at another important application of OEBs in error and convergence rate estimations. As pointed out in the theoretical proof of~\eqref{Rate-I} in Theorem~\ref{coro-I}, some numerical examples will be presented and analyzed in the next specified section. 

\subsubsection{Estimate of convergence rate}

At the heart of our convergence analysis is the rate of convergence that provides significant insights into numerical schemes. This section provides a novel analysis of Ishikawa iteration process~\eqref{xn-I} based on the idea of OEBs, as stated and proved in Theorem~\ref{coro-I}.

Going back to the test with Ishikawa scheme~\eqref{test-I} and two mappings $T_1, T_2$ as in~\eqref{T1T2}. Under the assumptions of Theorem~\ref{coro-I} and Remark~\ref{Rmk-theo-1}, we further consider $\delta$ and $\varepsilon$ are two positive constants such that
\begin{align*}
b_k \le \delta < \frac{1}{\alpha_2+1}; \quad 1 - b_k - \alpha_2 b_k (1-a_k+\alpha_1 a_k) \ge \varepsilon,
\end{align*}
for every $k \in \mathbb{N}$. More precisely, in our numerical test, we will set
\begin{align*}
\delta:= \max_{k \ge 0} b_k, \quad \varepsilon := \min_{k \ge 0} \big[1 - b_k - \alpha_2 b_k (1-a_k+\alpha_1 a_k)\big],
\end{align*}
and also introduce two prescribed parameters $\beta_{\min}$ and $\beta_{\max}$ as follows
\begin{align*}
\beta_{\min}: = 1 - \alpha_2, \quad  \beta_{\max} : = \min \left\{\frac{1+\alpha_2}{\varepsilon}; \frac{1+\alpha_2}{1-(1+\alpha_2)\delta}\right\}.
\end{align*}
As shown  in Theorem \ref{coro-I}, it allows us to conclude the estimate
\begin{align*}
\beta_{\min} \sum_{k=0}^n  b_k \le -\log(\mathrm{Err}_{n+1}) \le \beta_{\max} \sum_{k=0}^n  b_k, \quad \mbox{ for every } n \in \mathbb{N},
\end{align*}
which can be rewritten as
\begin{align}\label{eq:logerr}
\beta_{\min} \le \sigma_n := \frac{-\log(\mathrm{Err}_{n+1})}{\displaystyle\sum_{k=0}^n  b_k} \le \beta_{\max}, \quad \mbox{ for every } n \in \mathbb{N}.
\end{align}

Here, we employ the logarithmic scale to make presenting and estimating the convergence rate for error simpler. From~\eqref{eq:logerr}, to estimate the behavior of the ratio $\sigma_n$ as $n$ increases, we shall plot the error as a function of $\displaystyle\sum_{k=0}^n{b_k}$ on a log-scale. In addition, to verify the estimation in~\eqref{eq:logerr}, it is natural to include also two function plots of $Y = \beta_{\min}X$ (dotted blue lines) and $Y = \beta_{\max} X$ (black pluses lines), respectively. The idea behind using these plots is to control the convergence rate of the iterative process.
Here, we perform several numerical experiments to support theoretical results and to explain how the convergence rate can be estimated. In Figure~\ref{fig:5}, a log-scale plot of the convergence rates for different choices of two series $\sum{a_n}$ and $\sum{b_n}$ will be shown, where
\begin{align}\label{eq:n_a}
a_n = \frac{n+3}{2n+3}, \quad  b_n = \frac{1}{5},
\end{align}
and 
\begin{align}\label{eq:n_b}
a_n = \frac{2(n+1)^2+\cos(n+1)}{3(n+1)^2+2}, \quad  b_n = \frac{2(n+1)^2+1}{4(n+1)^3+\sin (n+1)},
\end{align}
respectively. Notice that in the plots of these two simulations, the observed convergence rates (dashed red lines) are always bounded between blue and black lines. These observation examples verified the theoretical assessment in Theorem~\ref{coro-I} empirically. 

\begin{figure}[H]
\centering
\subfigure[Case 1: $a_n$ and $b_n$ are chosen as in~\eqref{eq:n_a}.]{\includegraphics[width=0.49\linewidth]{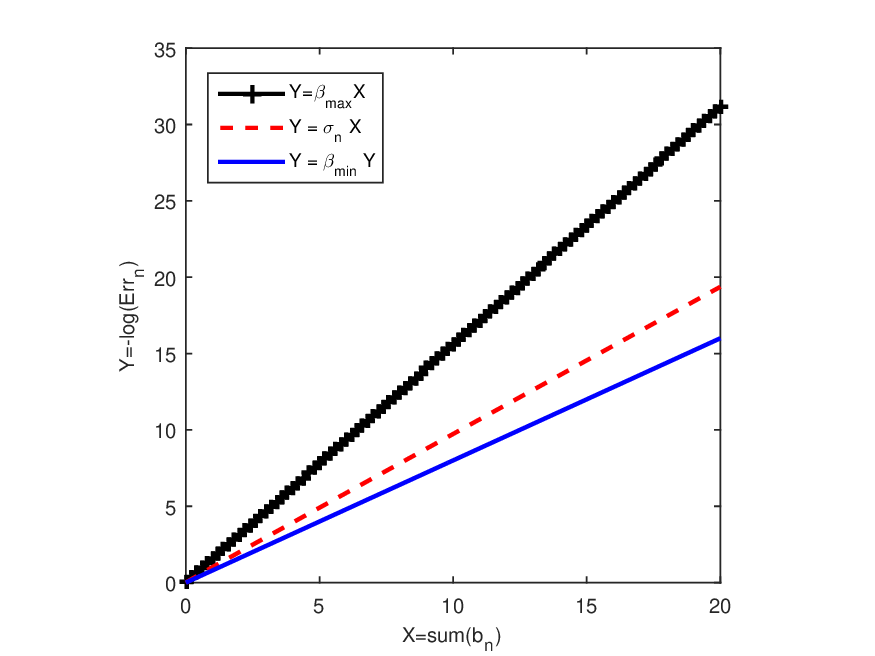}}
\subfigure[Case 2: $a_n$ and $b_n$ are chosen as in~\eqref{eq:n_b}.]{\includegraphics[width=0.49\linewidth]{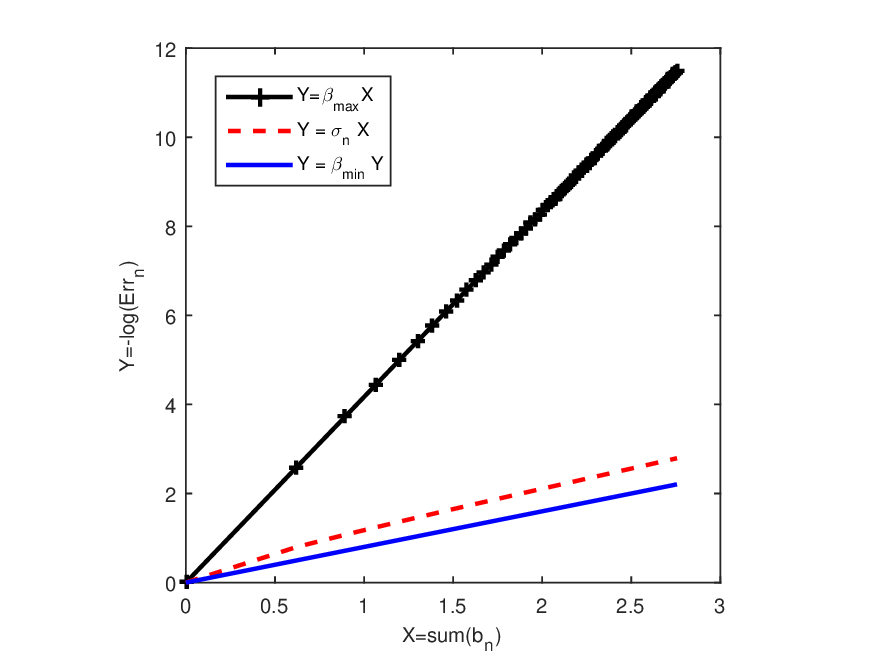}}
\caption{Estimate of convergence rates of Ishikawa iterative process~\eqref{xn-I} with $\alpha_1, \alpha_2$ as in~\eqref{alpha12}.}
\label{fig:5}
\end{figure}

\subsection{Modified Ishikawa iterative process}
\label{sec:mod_ishikawa_tests}
This section is devoted to presenting some numerical experiments performed with modified Ishikawa iteration process~\eqref{xn-IM-yn}. As what has been done with Ishikawa scheme~\eqref{xn-I}, the experiments will be shown to validate our theoretical results in Theorem~\ref{coro-IM}. 

Let us first rewrite the iterative model of~\eqref{xn-IM-yn} as follows
\begin{align}\label{test-IM}
\begin{cases} 
x_0 \in \Omega,\\
y_n = (1-a_n)x_n + a_nT_1(x_n), \\ x_{n+1} = (1-b_n)y_n + b_nT_2(y_n), \quad n \in \mathbb{N},\end{cases}
\end{align}
where $T_1, T_2$ are two non-expansive mappings defined as in~\eqref{T1T2}.

Since the overall numerical strategy is analogous to the previous section, we will only stress the key points and bring numerical tests to confirm necessary and sufficient conditions for the convergence of the modified sequence~\eqref{test-IM}. The present section is split into subsections, corresponding to subsequent steps toward the convergence results and estimation of convergence rate, respectively. 

\subsubsection{Convergence results}
Here, we verify the theoretical claims $(i)$ and $(ii)$ in Theorem~\ref{coro-IM}, where the convergence results of sequence $(x_n)$ to a common fixed point $x^*$ have been determined under assumptions of $\alpha_1, \alpha_2 \in (0,1)$ and the series $\sum{a_n}, \sum{b_n}$. Regarding this numerical example, we fix $\alpha_1$, $\alpha_2$ as in~\eqref{alpha12}, and four test cases of series $\sum{a_n}, \sum{b_n}$ as following:
\begin{itemize}
\item[(1)] Test 1: $\sum a_n = +\infty$ and $\sum b_n=+\infty$. For example, 
\begin{align*}
a_n = \frac{(n+1)^2+1}{3(n+1)^2+5}, \quad b_n = \frac{(n+1)^3+1}{4(n+1)^3-1}; \quad n \in \mathbb{N};
\end{align*}
\item[(2)] Test 2: $\sum a_n = +\infty$ and $\sum b_n<+\infty$. Let us choose 
\begin{align*}
a_n = \frac{1}{2n+5} \approx \frac{1}{n+1}, \quad b_n = \frac{\sin^4n}{(n+1)^2+5}; \quad n \in \mathbb{N};
\end{align*}
\item[(3)] Test 3: $\sum a_n < +\infty$ and $\sum b_n = +\infty$. We can choose
\begin{align*}
a_n = \frac{|\sin(n+1)|}{(n+1)^{\frac{4}{3}}+2}, \quad b_n = \frac{n+2}{3(n+1)^2+4}; \quad n \in \mathbb{N};
\end{align*}
\item[(4)] Test 4: $\sum a_n<+\infty$ and $\sum b_n <+\infty$. We consider here
\begin{align*}
a_n = \frac{1}{2(n+1)^{\frac{3}{2}}+5}, \quad b_n = \frac{1}{(n+1)^2+1}, \quad n \in \mathbb{N}.
\end{align*}
\end{itemize}

\begin{figure}[H]
\centering
\subfigure[Number of iterations $N=20$.]{\includegraphics[width=0.49\linewidth]{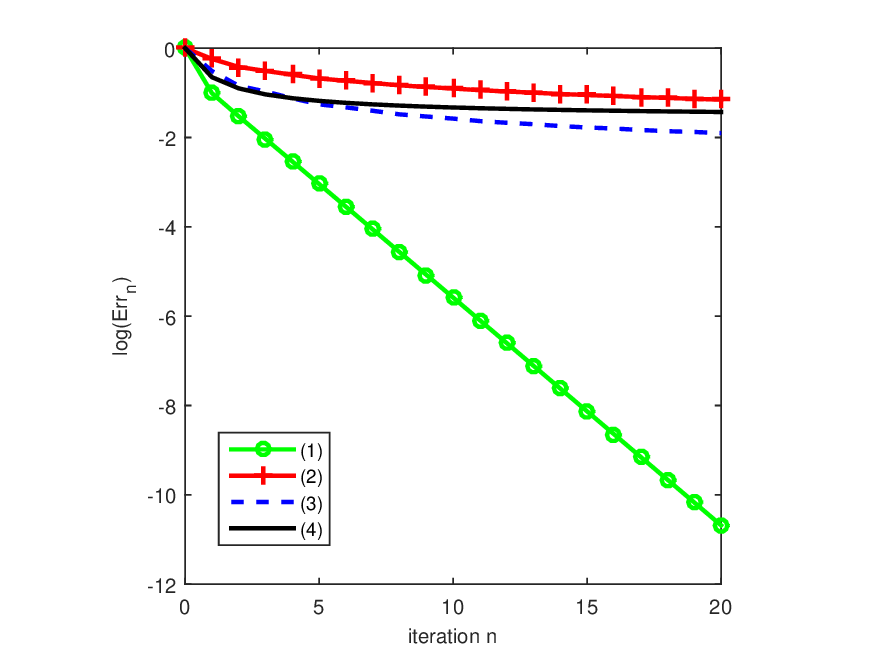}}
\subfigure[Number of iterations $N=10^5$.]{\includegraphics[width=0.49\linewidth]{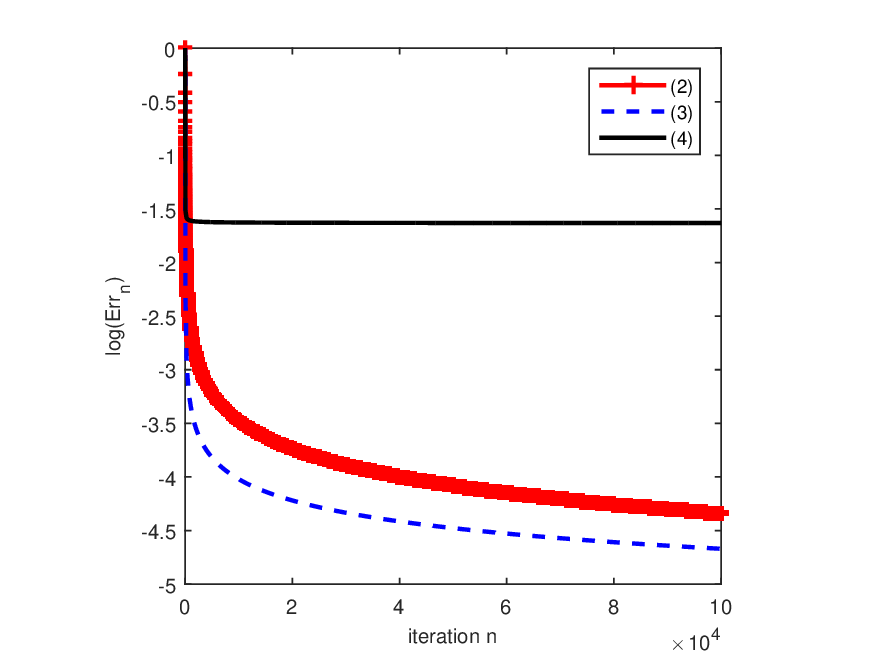}}
  \caption{Errors plotted against $n$ on a log-scale with four test models~(1)-(4).}
\label{fig:4}
\end{figure}

In Figure~\ref{fig:4}, convergence plots comparing the error as~\eqref{def:err} committed by our numerical tests on log-scale with dotted green, red (with plus signs), dashed blue, and dotted black lines indicating four cases presented above, respectively. To get a better observation, we show a logarithmic scale of error varying in the max number of iterations $N=20$ and $N=10^5$ in Figure~\ref{fig:4}(a) and~(b). It is observed that two test cases~(1)-(3) confirm our theoretical convergence analysis empirically, meanwhile, the last model fails to converge. Specifically, these results justify the observed fast convergence in the first model~(1) where $\sum (a_n + b_n)$ diverges fastest.

\subsubsection{Estimate of convergence rate}

Estimating the rate of convergence is also important in concluding the efficiency of an iterative scheme. As shown in Theorem~\ref{coro-IM}, it can be seen that the study of OEBs has a significant application in convergence rate estimation. In this section, these numerical experiments will be performed to verify that the observed convergence rates agree with our theory of OEBs. Analogously to the previous scheme, we recall the scheme~\eqref{test-IM}, under the assumptions of Theorem~\ref{coro-IM}, and also consider $\varepsilon_1$, $\varepsilon_2$ two positive constants such that
\begin{align*}
\varepsilon_1 = \min_{k \ge 0} (1-a_k-\alpha_1 a_k)>0, \quad \varepsilon_2 = \min_{k \ge 0} (1-b_k-\alpha_2 b_k)>0.
\end{align*}
The convergence rate of iterative process~\eqref{test-IM} therefore can be represented by the following ratio
\begin{align*}
\sigma_n := \frac{-\log(\mathrm{Err}_{n+1})}{\displaystyle\sum_{k=0}^n  (a_k+b_k)}, \quad n \in \mathbb{N}.
\end{align*}
Moreover, Theorem \ref{coro-IM} gives us the following estimate
\begin{align*}
\beta_{\mathrm{min}}  \le \sigma_n \le \beta_{\mathrm{max}}, \quad \mbox{ for every } n \in \mathbb{N},
\end{align*}
where $\beta_{\mathrm{min}}$ and $\beta_{\mathrm{max}}$ are defined by
\begin{align*}
\beta_{\mathrm{min}} := \min\{1-\alpha_1,1-\alpha_2\}; \quad \beta_{\mathrm{max}} := \min\left\{\frac{1+\alpha_1}{\varepsilon_1},\frac{1+\alpha_2}{\varepsilon_2}\right\}. 
\end{align*}

The strategy we follow to present convergence rates is the same as in the previous section with the Ishikawa scheme. We shall plot the error as a function of $\displaystyle\sum_{k=0}^n \big(a_k+b_k\big)$ on a log-scale. Also, to confirm the theoretical bounds of error, we respectively include here two function plots of $Y = \beta_{\mathrm{min}}X$ (dotted blue lines) and $Y = \beta_{\mathrm{max}} X$ (dashed black lines). Figure~\ref{fig:6} shows us the estimate of convergence rates when choosing $a_n$, $b_n$ as in~\eqref{eq:n_a} and
\begin{align}\label{eq:n_a2}
a_n = \frac{\sqrt{2(n+1)^2+1}}{4(n+1)^2+3}, \quad  b_n = \frac{1}{\sqrt{2n+3}}, \quad n\in \mathbb{N},.
\end{align}
One can observe that the convergence behavior is expected from our theoretical findings and the rate (dashed red line) can be estimated by comparing with lines of slope $\beta_{\mathrm{min}}$ (dotted blue) and $\beta_{\mathrm{max}}$ (dotted black). 

\begin{figure}[H]
\centering
\subfigure[$a_n$, $b_n$ are chosen as in~\eqref{eq:n_a}.]{\includegraphics[width=0.49\linewidth]{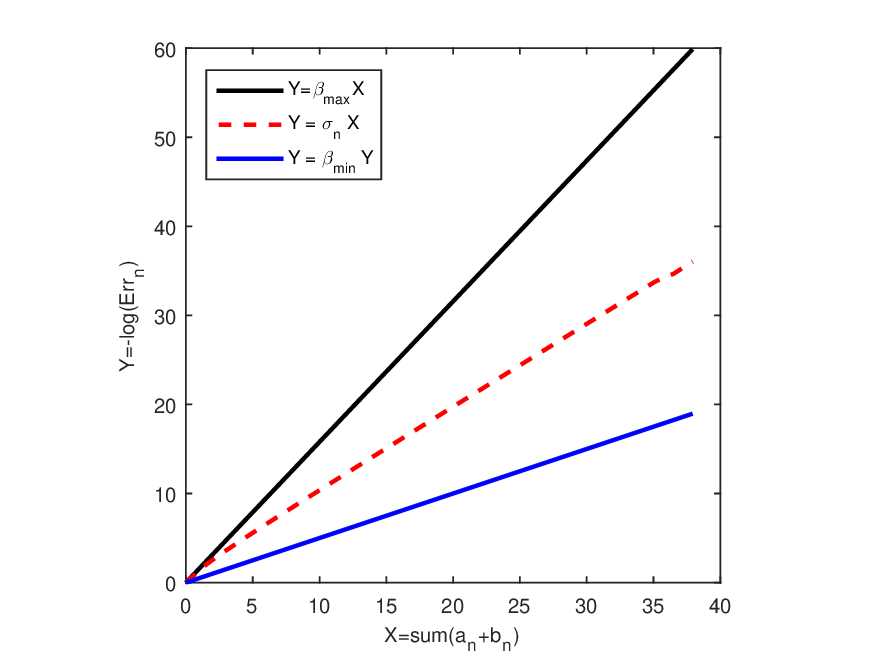}}
\subfigure[$a_n$, $b_n$ are chosen as in~\eqref{eq:n_a2}.]{\includegraphics[width=0.49\linewidth]{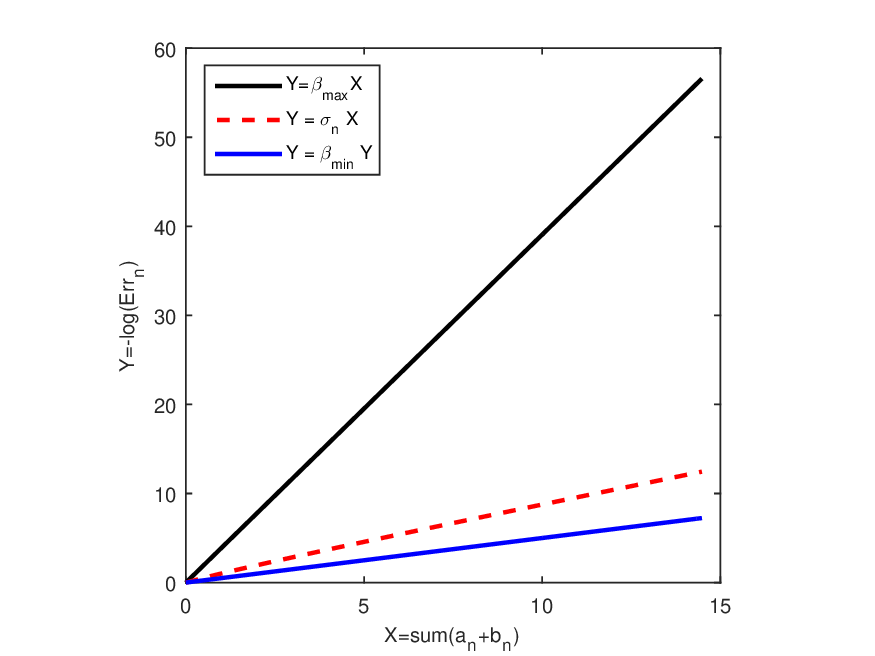}}
\caption{Estimate of convergence rate of modified Ishikawa iterative process~\eqref{xn-IM-yn} with $\alpha_1$, $\alpha_2$ in~\eqref{alpha12}.}
\label{fig:6}
\end{figure}

Having arrived at this stage, a question about the influence of OEBs in iterative schemes may arise: How to compare the iterative methods with the reasonable role of OEBs via convergence rate estimate? As we shall see in the next section, a variety of numerical experiments will be conducted to validate our theoretical findings in Theorem~\ref{theo:I-IM}.

\subsection{Comparison of convergence speeds}
In this section, we focus on some comparison experiments between two considered iterative schemes~\eqref{xn-I} and~\eqref{xn-IM-yn}. Here, we perform some numerical examples to verify the theory proved in Theorem~\ref{theo:I-IM}. To do that, let us consider sequences $(x_n^{\mathrm{I}})$ and $(x_n^{\mathrm{IM}})$ presented by two processes~\eqref{xn-I} and~\eqref{xn-IM-yn} to approximate a common fixed point $x^*$ of two non-expansive mappings $T_1$ and $T_2$ given in~\eqref{T1T2}. Further, we fix two values $\alpha_1$ and $\alpha_2$ as in~\eqref{alpha12}. In the sequel, to illustrate the numerical results, we shall also present the executive errors on a logarithmic scale against the number of iterations $n$.

In the first test, error plots for two simulations~\eqref{xn-I} and~\eqref{xn-IM-yn} are shown when $(a_n)$, $(b_n)$ not satisfying~\eqref{cond:I-IM-a} as following
\begin{itemize}
\item[(1)] Test 1: Consider $a_n \approx \frac{1}{n+1}$ and $b_n \approx \frac{1}{5}$, in particular one may choose
\begin{align}\notag 
a_n = \frac{2n+1}{(2n+3)^2}, \quad b_n = \frac{n+2}{5n+9},\quad n \in \mathbb{N};
\end{align}
\item[(2)] Test 2: Consider $a_n \approx \frac{1}{(n+1)^2}$ and $b_n \approx \frac{1}{\sqrt{n+1}}$, we specifically choose
\begin{align}\notag 
a_n = \frac{\sqrt{2(n+1)^2+1}}{4(n+1)^3+5}, \quad b_n = \frac{|\sin (n+1)|}{\sqrt{6n+5}}, \quad n \in \mathbb{N}.
\end{align}
\end{itemize}

In Figure~\ref{fig:8a}(a) and (b), we establish comparison results of convergence rates between two schemes: Ishikawa~\eqref{xn-I} (dotted red lines) and the modified one~\eqref{xn-IM-yn} (dashed black lines). It can be seen that the error of~\eqref{xn-IM-yn} always has better behavior than~\eqref{xn-I} whenever the convergence result holds. This coincides with the statement in \eqref{IMfasterI}. However, it is not sufficient to conclude that \eqref{xn-IM-yn} converges faster than~\eqref{xn-I} in the sense of~\eqref{faster} in Definition~\ref{def-conv-rate}. The reason is that $(a_n)$, $(b_n)$ do not satisfy~\eqref{cond:I-IM-a}. Therefore, the ratio $\mathcal{R}(x^{\mathrm{IM}}_n,x^{\mathrm{I}}_n,x^*)$ may not converge to zero. To clarify this fact, we plot the following term
$$\log\mathcal{R}\big(x_n^{\mathrm{IM}},x_n^{\mathrm{I}},x^*\big) = \log \left( \frac{\mathrm{Err}_n^{\mathrm{IM}}}{\mathrm{Err}_n^{\mathrm{I}}}\right)$$ 
between two iterative sequences with respect to the iteration $n$, where $\mathrm{Err}_n^{\mathrm{I}}$ and $\mathrm{Err}_n^{\mathrm{IM}}$ denote the executive errors of two iterative processes~\eqref{xn-I} and~\eqref{xn-IM-yn}, respectively. Figure~\ref{fig:8a}(c) and (d) show that $\mathcal{R}\big(x_n^{\mathrm{IM}},x_n^{\mathrm{I}},x^*\big)$ converges to a positive constant.

\begin{figure}[H]
\centering
\subfigure[Comparison of numerical convergence rates in Test 1.]{\includegraphics[width=0.49\linewidth]{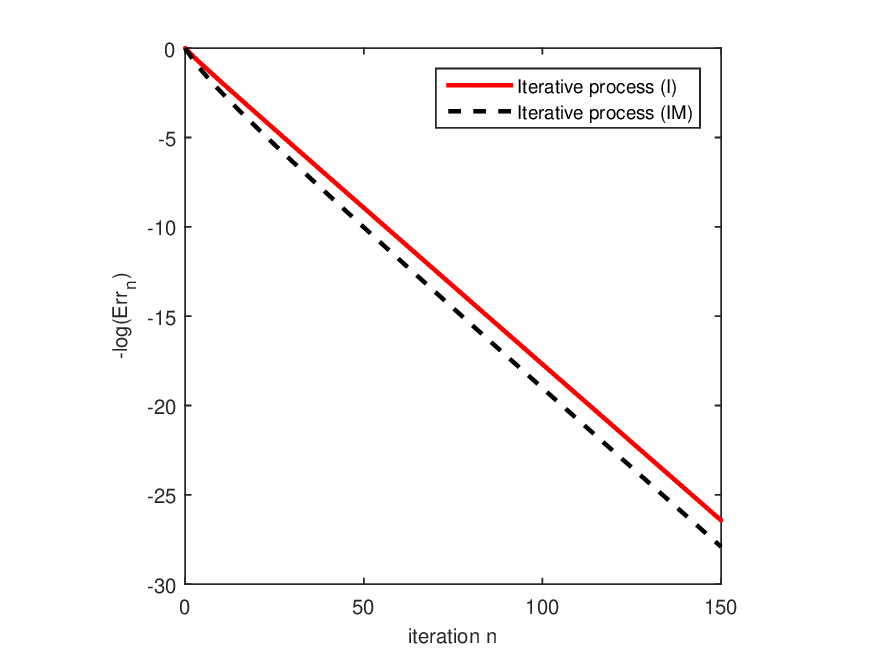}}
\subfigure[Comparison of numerical convergence rates in Test 2.]{\includegraphics[width=0.49\linewidth]{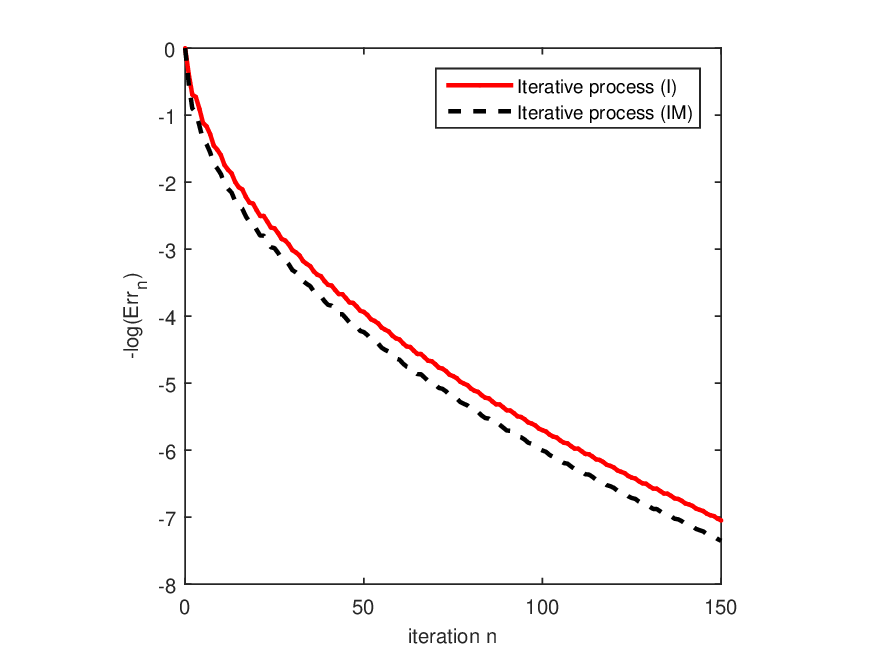}}
\subfigure[$\log\mathcal{R}\big(x_n^{\mathrm{IM}},x_n^{\mathrm{I}},x^*\big)$ for Test 1.]{\includegraphics[width=0.49\linewidth]{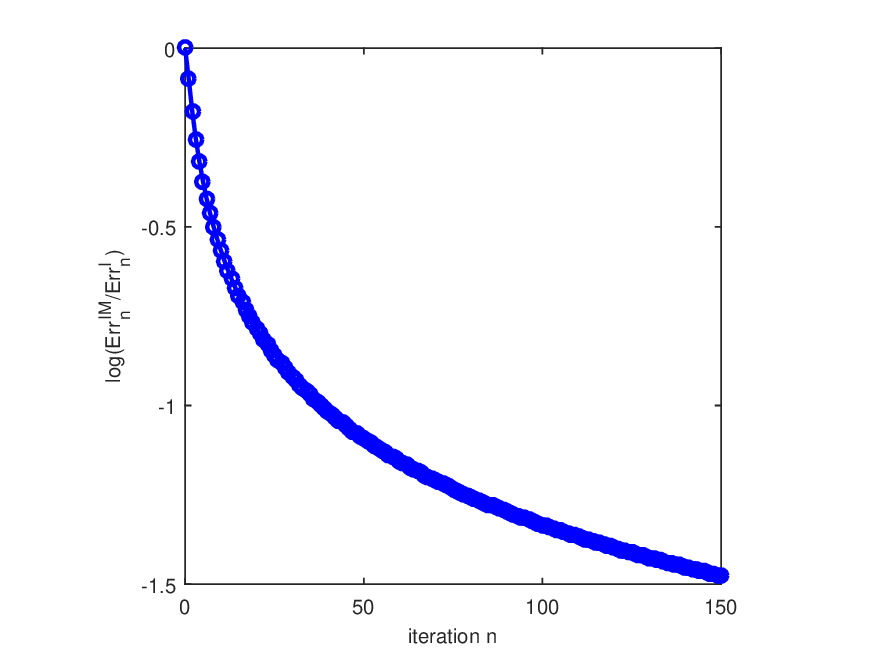}}
\subfigure[$\log\mathcal{R}\big(x_n^{\mathrm{IM}},x_n^{\mathrm{I}},x^*\big)$ for Test 2.]{\includegraphics[width=0.49\linewidth]{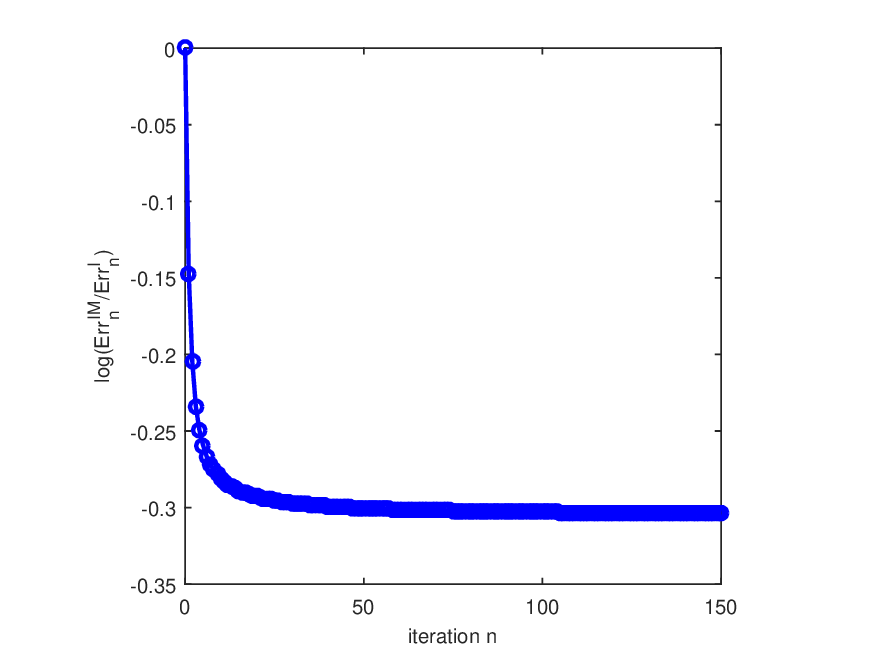}}
\caption{Comparison of convergence rates between two iterative schemes~\eqref{xn-I} and~\eqref{xn-IM-yn} without assumption~\eqref{cond:I-IM-a}.}
\label{fig:8a}
\end{figure}

In the next numerical test, we shall choose $(a_n)$, $(b_n)$ satisfy the assumption~\eqref{cond:I-IM-a}, this is equivalent to
\begin{align}\label{cond-antest}
a_n \ge \frac{(1+\alpha_2)b_n}{(1-\alpha_1)(1+\alpha_2b_n)}, \quad \mbox{ for all } n  \in \mathbb{N}.
\end{align}
At this stage, thanks to the condition $a_n \le 1$, we infer that
\begin{align*}
b_n \le \frac{1-\alpha_1}{1+\alpha_1\alpha_2}, \quad \mbox{ for all } n  \in \mathbb{N}.
\end{align*}
For this reason, to ensure the validation of condition \eqref{cond-antest}, we first take $(b_n)$ such that 
$$\sum b_n = \infty, \ \ \text{and} \ \ b_n \le \frac{1-\alpha_1}{1+\alpha_1\alpha_2}. \quad n  \in \mathbb{N},$$ 
Then, it remains to choose $(a_n)$ accordingly to
\begin{align}\label{an-in-bn}
a_n = \min\left\{1; \ \mathrm{rand}([0,1]) + \frac{(1+\alpha_2)b_n}{(1-\alpha_1)(1+\alpha_2b_n)}\right\}, \quad \mbox{ for all } n \in \mathbb{N}.
\end{align}

To highlight the numerical experiments, we perform two following test models
\begin{itemize}
\item[(3)] Test 3: $\displaystyle b_n = \frac{n+2}{5n+9},\quad n \in \mathbb{N};$
\item[(4)] Test 4: $\displaystyle b_n = \frac{|\sin (n+1)|}{\sqrt{6n+5}}, \quad n \in \mathbb{N},$
\end{itemize}
along with two randomly sequences $(a_n)$ determined in \eqref{an-in-bn}, respectively. Analogously to Figure~
\ref{fig:8a}, Figure~\ref{fig:compare} shows the comparison of convergence rates between two iterative schemes fulfill assumption~\eqref{cond:I-IM-a}. In Figure~\ref{fig:compare}(a) and (b), we only present the rate at which the iterative process converges more slowly (regarding the Ishikawa iteration~\eqref{xn-I}). In Figure~\ref{fig:compare}(c) and (d), we plot the logarithm of ratio $\mathcal{R}\big(x_n^{\mathrm{IM}},x_n^{\mathrm{I}},x^*\big)$ against the iteration $n$ to derive an observation of comparison strategy.  
 
\begin{figure}[H]
\centering
\subfigure[Numerical convergence rate of \eqref{xn-I} in Test 3.]{\includegraphics[width=0.49\linewidth]{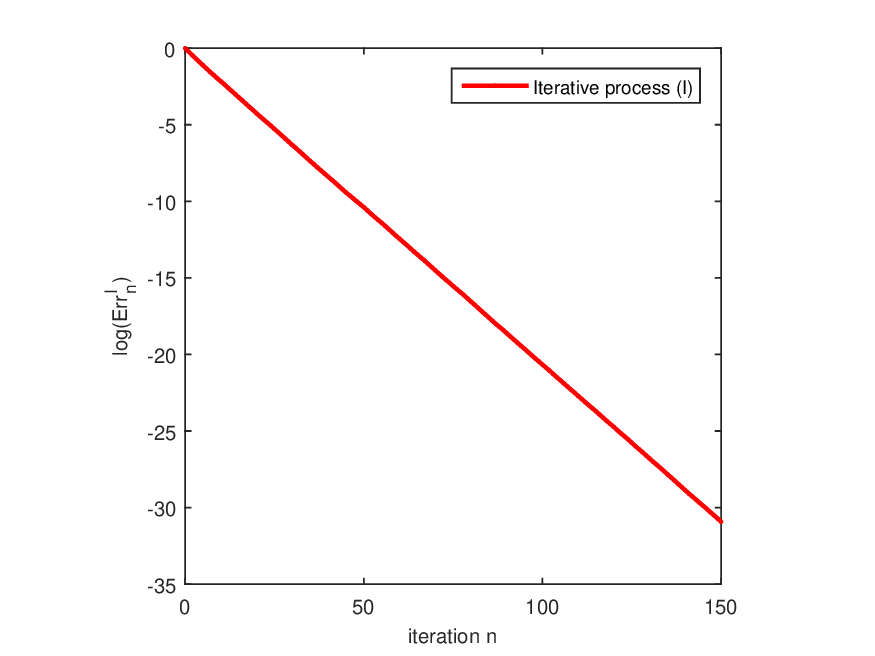}}
\subfigure[Numerical convergence rate of \eqref{xn-I} in Test 4.]{\includegraphics[width=0.49\linewidth]{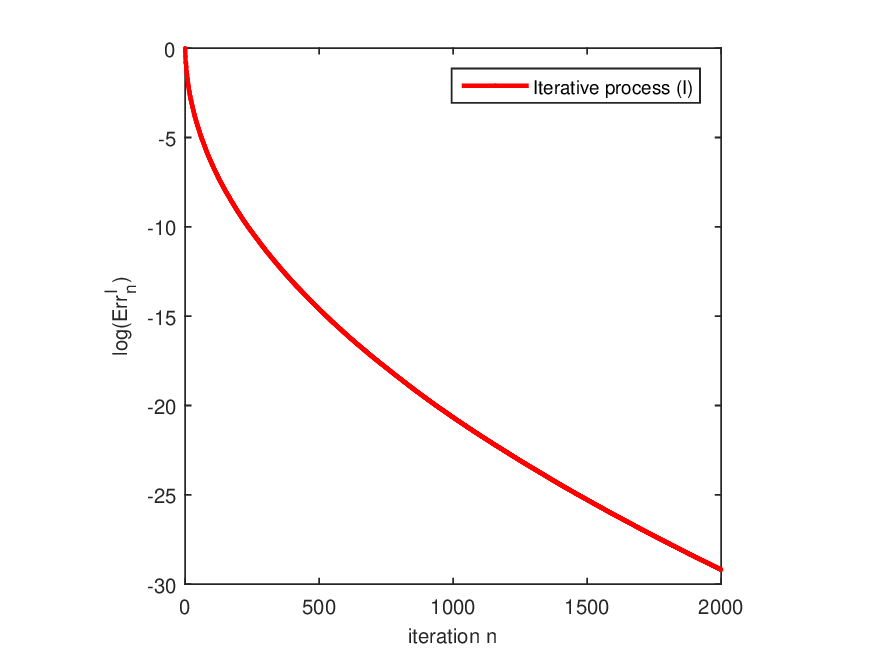}}
\subfigure[$\log\mathcal{R}\big(x_n^{\mathrm{IM}},x_n^{\mathrm{I}},x^*\big)$ for Test 3.]{\includegraphics[width=0.49\linewidth]{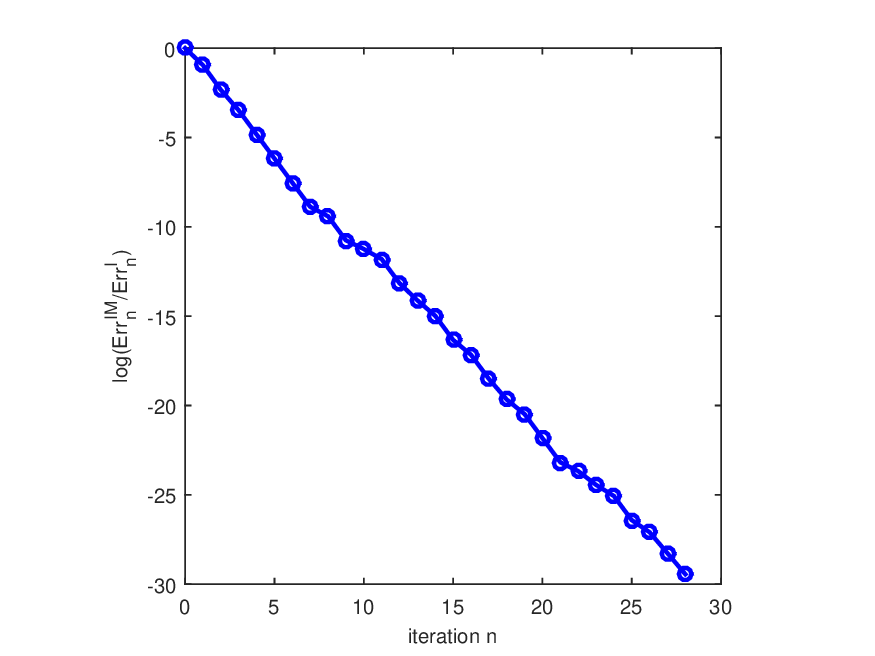}}
\subfigure[$\log\mathcal{R}\big(x_n^{\mathrm{IM}},x_n^{\mathrm{I}},x^*\big)$ for Test 4.]{\includegraphics[width=0.49\linewidth]{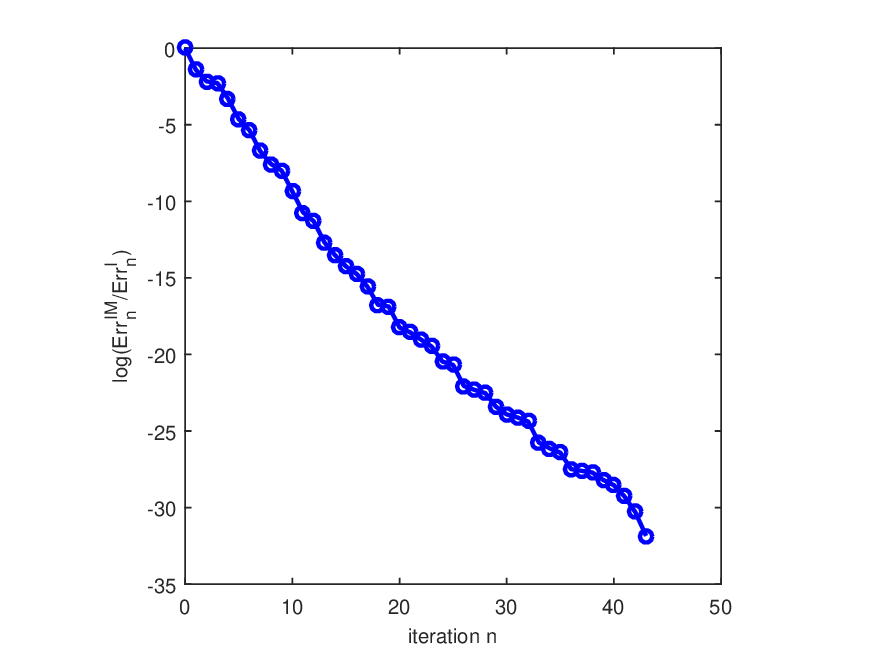}}
\caption{Comparison of convergence rates between two iterative schemes~\eqref{xn-I} and~\eqref{xn-IM-yn} under assumption~\eqref{cond:I-IM-a}.}
\label{fig:compare}
\end{figure}
From numerical results, it can be easily seen that
\begin{align*}
\mathcal{R}\big(x_n^{\mathrm{IM}},x_n^{\mathrm{I}},x^*\big) \to 0 \mbox{ as } n \to \infty,
\end{align*}
which in turn will allow us to conclude the modified sequence~\eqref{xn-IM-yn} converges more rapidly than the sequence in~\eqref{xn-I} in the sense of Definition~\ref{def-conv-rate}. The results confirm our theoretical analysis for convergence rate.


\section{Conclusion}
\label{sec:conclude}
\subsection*{Concluding remarks}
This paper constitutes an effort toward the convergence study of iterative schemes from a theoretical point of view. Here, we explore a new approach to convergence analysis based on the so-called \emph{optimal error bounds} (OEBs). As far as we know, this type of error bound has been less investigated in establishing the rate of convergence and convergence analysis of iterative methods. In the study, we state and prove some theoretical results regarding OEBs, that are useful for deriving necessary and sufficient conditions for the convergence and estimating the rate of convergence of an iterative sequence. To support our theoretical findings, we also carry out some numerical results in this work. Via a few numerical examples, it indicates that illustrations confirm our theoretical results in Theorem~\ref{coro-I}, ~\ref{coro-IM}. Moreover, in convergence rate comparison between different algorithms, the study of OBEs in Theorem~\ref{theo:I-IM} brings us many advantages. We also include some numerical tests as well as their interpretation in this paper to establish our claim in the main theorems. 
\subsection*{Significance and novelty of the study}
In this paper, a new theory is developed from the convergence rate point of view, where the study of iterative methods for finding a common fixed point of two non-expansive mappings can be presented as an example. We believe that it would be very interesting to exhibit and apply the study of OBEs to a large class of iterative models. Moreover, in numerical analysis, it is also interesting when studying (sharp) OEBs to investigate necessary and/or sufficient conditions under which we have the convergence. Our approach opens a new perspective on the study of the convergence rate estimate of fixed-point iterative methods and also illustrates the comparison strategies between methods. It can be more broadly applicable to developing convergence conditions for an abstract iterative scheme in various general models. 
\subsection*{Further research}
The approach with OEBs (both OUEB and OLEB) can be applied to a larger class of generalized operators to approximate the set of common fixed points. Further, it is also of theoretical and practical importance to conduct iterative schemes in the setting of metric spaces (endowed with graph structures), which will be detailed in our forthcoming paper. 
Not only deal with the Ishikawa and modified Ishikawa iteration processes but the idea of OEBs can also be extended to address the necessary/sufficient conditions for convergence and the convergence rate estimate of several generalized multi-step iterative algorithms.


\section*{Conflict of Interest}
The authors declared that they have no conflict of interest.

%
%
%

%
%
%

\end{document}